\documentclass[12pt]{article}

\usepackage{amsmath}
\usepackage{amsfonts}
\usepackage{amssymb}
\usepackage{amscd}
\usepackage{amsthm}
\usepackage{textcomp}
\usepackage{bbm}
\usepackage{color}
\usepackage[linktocpage=true,plainpages=false,pdfpagelabels=false]{hyperref}

\input xypic

\usepackage[matrix,arrow,curve]{xy}

\usepackage{color}

\newcommand{\quash}[1]{}


\numberwithin{equation}{section}

\newtheorem{defin}{Definition}[section]
\newtheorem{prop}{Proposition}[section]
\newtheorem{nt}{Remark}[section]
\newtheorem{Th}{Theorem}[section]

\newtheorem{cons}{Corollary}[section]

\newtheorem{defin-prop}{Definition-proposition}[section]

\newfont{\ssdbl}{msbm8}
\newfont{\sdbl}{msbm9}
\newfont{\dbl}{msbm10 at 12pt}

\newcommand{\oo}{{\cal O}}
\newcommand{\ff}{{\cal F}}

\newcommand{\g}{{\cal G}}

\newcommand{\res}{\mathop {\rm res}}
\newcommand{\Hom}{\mathop {\rm Hom}}
\newcommand{\Aut}{\mathop {\rm Aut}}

\newcommand{\tr}{\mathop {\rm tr}}

\newcommand{\Spec}{\mathop {\rm Spec}}

\newcommand{\dz}{\mathbb{Z}}
\newcommand{\dc}{\mathbb{C}}

\newcommand{\dr}{\mathbb{R}}

\newcommand{\Z}{\dz}

\newcommand{\ph}{\varphi}

\newcommand{\Ker}{{\rm Ker}\:}
\newcommand{\Image}{{\rm Im}\:}

\newcommand{\lrto}{\longrightarrow}

\newcommand{\M}{{\cal M}}

\def\Z{{\mathbb Z}}
\def\Q{{\mathbb Q}}

\newcommand{\LL}{{\mathcal L}}
\newcommand{\RR}{{\mathcal R}}
\newcommand{\AutL}{{   {\mathcal Aut}^{\rm c, alg} ({\mathcal L} )}}
\newcommand{\AutLm}{{   {\mathcal Aut}_{-}^{\rm c, alg} ({\mathcal L} )}}

\newcommand{\AutLmo}{{   {\mathcal Aut}_{-,1}^{\rm c, alg} ({\mathcal L} )}}
\newcommand{\AutLp}{{   {\mathcal Aut}_{+}^{\rm c, alg} ({\mathcal L} )}}
\newcommand{\AutLpo}{{   {\mathcal Aut}_{+,1}^{\rm c, alg} ({\mathcal L} )}}

\newcommand{\Lie}{\mathop {\rm Lie}}

\newcommand{\GG} {{\mathbb G}}
\newcommand{\CC}{\mathop{\rm CC}}
\newcommand{\GL}{{\mathop{\rm GL}}}

\newcommand{\ve}{\varepsilon}

\begin{document}

\author{
D. V. Osipov
}

\title{Formal Bott-Thurston cocycle and part of a formal Riemann-Roch theorem  \thanks{The author was supported by the Basic Research Program of the National Research University Higher
School of Economics.}}
\date{}

\maketitle

{\begin{center} {\em  Dedicated to the memory of  A.~N.~Parshin}
\end{center}}

\begin{abstract}
The Bott-Thurston cocycle is a  $2$-cocycle on the group of orientation-preserving diffeomorphisms of the circle.
We introduce and study a formal analog of Bott-Thurston cocycle. The formal Bott-Thurston cocycle is a $2$-cocycle on the group of continuous $A$-automorphisms of the algebra $A((t))$ of Laurent series over a commutative ring $A$ with values in the group $A^*$ of invertible elements of $A$.
We prove that the  central extension given by the formal Bott-Thurston cocycle  is equivalent to the $12$-fold Baer sum of the determinantal central extension when $A$ is a $\Q$-algebra.
As a consequence of this result we prove a part of new formal Riemann-Roch theorem. This Riemann-Roch theorem is applied to  a ringed space on a separated scheme $S$ over $\Q$, where the structure sheaf of the ringed space is locally on $S$ isomorphic to the sheaf $\oo_S((t))$ and the transition automorphisms are continuous. Locally on $S$ this ringed space corresponds to the punctured formal neighbourhood of a section of a  smooth morphism to $U$ of relative dimension $1$, where an open subset $U \subset S$.
\end{abstract}

\section{Introduction}

The goal of this paper is to relate the Bott-Thurston cocycle and a part of the Riemann-Roch theorem at the formal level.

\subsection{Bott-Thurston cocycle}  \label{BT}

We first describe, what is the Bott-Thurston cocycle. We recall the following construction from Bott's paper~\cite{B1}.
Let $M$ be a smooth compact oriented $n$-dimensional manifold, and $\Gamma$ be a group which acts on the right on $M$ by orientation-preserving diffeomorphisms. We consider the Abelian group $\dr_+$ of positive real numbers with respect to multiplication.

We note that because we supposed that $\Gamma$ acts on $M$ on the right, $\Gamma$ acts on the Abelian group $C^{\infty}(M, \dr_+)$ of smooth functions from $M$ to $\dr_+$  on the left by the rule
$$
h(f)(x )=f (x h)  \, \mbox{,}
$$
where $h \in \Gamma$, $f \in C^{\infty}(M, \dr_+)$, $x \in M$.

Let $\omega$ be a volume form on $M$. For any $h \in \Gamma$ we define  $\mu(h)  \in C^{\infty}(M, \dr_+)$  in the following way
\begin{equation}  \label{mu}
\mu(h) = \frac{h^*(\omega)}{ \omega}    \mbox{.}
\end{equation}

Now, an $n+1$-cocycle $B$ such  that its class is  an element from the group $H^{n+1}(\Gamma, \dr)$ is given as
\begin{equation}  \label{Bo}
B(h_1, \ldots, h_{n+1}) = \int_M \log  \mu(\widehat{h_1})  \  d \log \mu(\widehat{h_2})  \wedge  \ldots  \wedge d \log \mu(\widehat{h_{n+1}}) \, \mbox{,}
\end{equation}
where $h_1, \ldots, h_{n+1}  \in \Gamma$, and $\widehat{h_i}= h_1 \ h_2 \ldots h_i$ for $1 \le i \le n$.

From \eqref{mu} it is easy to obtain that  for any $h_1, h_2 \in \Gamma$
\begin{equation}  \label{chain}
\mu (h_1 h_2)= \mu(h_1) \cdot h_1(\mu(h_2))
\end{equation}
which can be seen as the chain rule for differentation. We note that for~\eqref{chain} it is important that $\Gamma$ acts on $M$ on the right.
Besides, formula~\eqref{chain} is the formula for an $1$-cocycle in the group cohomology, i.e. $\mu$ is an $1$-cocycle and its class defines an element from $H^1(\Gamma, C^{\infty}(M, \dr_+))$.

In this paper we will be interested in the formal analog of the cocycle $B$ when $n=1$ and $M =S^1$ is the circle, but see Remark~\ref{higher} below. Then in this situation the cocylce $B$ is called
the Bott-Thurston cocycle (when $\omega = d \theta$, where $\theta$ is a coordinate on the circle $S^1 = \dr/ 2\pi \dz$), and the central extension of the group $\Gamma= {\rm Diff}(S^1)$ of orientation-preserving diffeomorphisms of $S^1$  by the group $\dr$ given by this cocycle is called the Bott-Virasoro group, see, e.g.,~\cite{GR,KW}.

We change slightly the Bott-Thurston cocycle $B$ to the cocycle $\widetilde{B}$ with values in the subgroup $\exp (i \dr) \subset \dc^*$:
$$
\widetilde{B}= \exp \frac{B}{2\pi i} \, \mbox{.}
$$
Then, using formula~\eqref{chain}, we have
$$
\widetilde{B} (h_1, h_2) = \exp \frac{1}{2\pi i} \int_{S^1} \log  \mu(\widehat{h_1})  \  d \log \mu(\widehat{h_2})) =
\exp \frac{1}{2\pi i} \int_{S^1} \log  \mu({h_1})  \ d \log h_1 (\mu(h_2))  \, \mbox{.}
$$
where $h_1, h_2 \in {\rm Diff}(S^1)$. (We recall that we supposed that ${\rm Diff}(S^1)$ acts on $S^1$ on the right.)

\subsection{Interpretation of cocycle via product in group cohomology}  \label{prod-gr-coh}

From the point of view of group cohomology and its product the  cocycle $\widetilde{B}$ can be described in the following way.

We suppose that a group $\Gamma$ acts as usual, i.e. on the left, on an Abelian group $K$, and we will use multiplicative notation for the group laws. We have the group
$H^1(\Gamma, K)$ that is the quotient group of the group of $1$-cocycles by the subgroup of $1$-coboundaries, where $1$-cocycles
$\lambda : \Gamma \to K$ are described by the property:
$$
\lambda(h_1 h_2)= \lambda(h_1) \ h_1(\lambda(h_2))  \, \mbox{,}
$$
where $h_i \in \Gamma$ for $i =1$ and $i =2$.

Then there is a well-defined homomorphism, which is called  $\cup$-product, see~{\cite[ch.~V, \S~3]{Bro}}:
$$
\cup \; :  \; H^1(\Gamma, K)  \otimes_{\Z} H^1(\Gamma, K) \lrto H^2(\Gamma, K \otimes K)
$$
given by means of the tensor product of the standard resolutions (or, in other words, bar-resolutions) and  the diagonal approximation given by the  Alexander-Whitney map, and where $\Gamma$ acts diagonally on $K \otimes_{\Z} K$. Explicitly,  the $\cup$-product gives the  $2$-cocycle from   two $1$-cocycles as follows:
$$
\lambda_1 \otimes \lambda_2   \longmapsto \left\{ (g,h) \mapsto \lambda_1(g)  \otimes g(\lambda_2(h)  \right\}  \, \mbox{,}
$$
where $\lambda_i$ is a $1$-cocycle from $\Gamma$ to $K$  ($i=1$ and $i=2$) and $g,h \in \Gamma$.

Now, suppose that we have  a  homomorphism of Abelian groups
$$
C \; : \; K \otimes_{\Z} K \lrto E
$$
that commutes with the action of $\Gamma$ on $K \otimes_{\Z} K$ and with the trivial action of $\Gamma$ on $E$. Then we obtain a well-defined map
given by the following composition of maps
\begin{equation}  \label{cup-prod}
H^1 (\Gamma, K)  \otimes_{\Z} H^1(\Gamma, K)  \lrto H^2(\Gamma, K \otimes_{\Z} K )  \lrto H^2(\Gamma, E)  \, \mbox{,}
\end{equation}
where the first map is the $\cup$-product and the second map is induced by $C$. Besides, map~\eqref{cup-prod} is written
by means of map on corresponding cocycles.

In the case of the cocycle $\widetilde{B}$ we have $\lambda_1 = \lambda_2 = \mu $ and a map $C$ is a map
\begin{equation}  \label{map-C}
 f_1 \otimes f_2  \longmapsto  \exp \frac{1}{2\pi i} \int_{S^1} \log f_1  \ d \log f_2 \, \mbox{,}
\end{equation}
where $f_1, f_2  \in C^{\infty}(M, \dr_+)$. Clearly, this map is ${\rm Diff}(S^1)$-equivariant.

\subsection{Similarity with the Deligne-Riemann-Roch theorem}

Now we describe a part of the  Deligne-Riemann-Roch theorem, whose formal analog we will consider.

Let $\pi : X \to S$ be a family of smooth proper curves over a scheme $S$, i.e., $\pi$ is a smooth proper morphism of relative dimension~$1$. Consider an invertible sheaf $\omega = \Omega^1_{X/S}$  on $X$.
Then there is a canonical isomorphism of invertible sheaves on $S$ (see~\cite{D1})
\begin{equation}  \label{DRR}
(\det R \pi_* \oo_X)^{\otimes 12}  \cong  \langle \omega, \omega \rangle  \, \mbox{,}
\end{equation}
where $\langle \cdot, \cdot \rangle$  is the Deligne bracket (see \cite{D1}), which assigns an invertible sheaf on
$S$ to two invertible sheaves on $X$. We note (assuming  $S$ is a smooth algebraic variety over a field) that inside the group ${\rm Pic}(S)$ the Deligne bracket for invertible sheaves $L$ and $M$ on $X$ is
\begin{equation}  \label{DBr}
c_1(\langle L, M \rangle) = \pi_* (c_1(L)  \cdot c_1(M))  \, \mbox{,}
\end{equation}
where $\pi_*$  is the direct image, which is the homomorphism
\begin{equation}  \label{dir-im}
{\rm CH}^2(X) = H^2(X, K_2(\oo_X))  \lrto {\rm Pic}(S)= H^1(S, \oo_S^*)  \, \mbox{.}
\end{equation}

Now we can compare formulas~\eqref{cup-prod} and~\eqref{DBr}.
Besides, the map $C$ given by formula~\eqref{map-C}  will look on the formal level as
\begin{equation}  \label{formCC}
\exp \res  \left( \log( \cdot) \, d \log(\cdot) \right) \, \mbox{,}
\end{equation}
where we changed $\frac{1}{2\pi i} \int_{S^1}$ to the residue  $\res$. Formula~\eqref{formCC} looks as the partial formula for the  Contou-Carr\`{e}re symbol  $\CC$, see section~\ref{Contou} below. But it is known that the Contou-Carr\`{e}re symbol $\CC$ gives the direct image, see~\cite[Th.~2, Prop.16]{Osip1},   \cite{KV}, \cite[Prop.~8.28, Rem.~8.30]{GO1}.
This all gives  the similarity of the cocycle $\widetilde{B}$ with the right hand side of formula~\eqref{DRR}.

We also note  that in~\S~1.5 of Introduction to \cite{BKTV}  there was a mention (but without any further development and over the field $\dr$) that group cocycles introduced by R.~Bott
in~\cite{B1}, see formula~\eqref{Bo},   should be   related to a side of a group theoretical Real Riemann-Roch theorem.

\subsection{Part of new formal Riemann-Roch theorem}
Now we describe the formal analog of equality~\eqref{DRR}.

Let $A$ be any commutative ring, and let $A((t))$ be the Laurent series ring over $A$. The ring $A((t))$ has the natural topology. Let
$\AutL(A)$ be the group of continuous automorphisms of the $A$-algebra $A((t))$. Then any element $\ph \in \AutL(A)$   is uniquely defined
by the element $\widetilde{\ph} = \ph(t) \in A((t))^*$  (see more in Section~\ref{Autgr}).

Now the formal Bott-Thurston cocycle $\hat{B}$, which is a formal analog of the cocycle $\widetilde{B}$,  is a $2$-cocycle on the group $\AutL(A)$
with coefficients in $A^*$. It looks as follows
$$
\hat{B} (\ph_1, \ph_2) = \CC (\widetilde{\varphi_1}', \,  \widetilde{\varphi_2}'\circ \widetilde{\varphi_1})  \, \mbox{,}
$$
where ${}'$ is the partial derivative with respect to $t$, and   $\widetilde{\varphi_2}'\circ \widetilde{\varphi_1}$ is the Laurent series from $A((t))$ that is the result of the substitution
of the series $\widetilde{\varphi_1}$ into the series $\widetilde{\varphi_2}'$ instead of variable $t$.

On the other hand, there  is the determinantal central extension $\widetilde{\AutL}(A)$ of the group $\AutL(A)$
by the group $A^*$. The group $\widetilde{\AutL}(A)$ consists of pairs $(g,s)$, where $g \in \AutL(A)$ and $s$
is an element of free $A$-module  $\det( g (A[[t]])  \mid  A[[t]])$ of rank $1$
such that for any prime ideal  $P$ of the ring  $A$ we have   $s \notin P \det(g (A[[t]])  \mid  A[[t]])$.
Here $\det(g (A[[t]])  \mid  A[[t]])$  is the relative determinant of $A$-modules $g (A[[t]])$ and $ A[[t]]$, which is  canonically isomorphic to the $A$-module
$$
\Hom\nolimits_A \left( \bigwedge^{\rm max} (g(A[[t]])/ t^l A[[t]]), \,  \bigwedge^{\rm max} (A[[t]]/ t^l A[[t]])                   \right)      \, \mbox{,}
$$
where an integer $l$ satisfies the condition  $t^l A[[t]]   \subset (A[[t]] \cap g(A[[t]]))$.

We note that the determinantal central extension originates (when $A$ is a field) from the paper~\cite{KP} of V.~G.~ Kac and D.~H.~Peterson.

There is a natural non-group section from $\AutL(A)$ to $\widetilde{\AutL}(A)$. This section gives the $2$-cocycle $D$ on the group $\AutL(A)$
with coefficients in  the group~$A^*$:
$$
D(f,g ) =  \det (d_f \cdot d_g  \cdot d_{fg}^{-1}  )  \, \mbox{,}
$$
where $f,g \in \AutL(A)$, and for any $h \in \AutL(A)$ the map
$$ d_h = \mathop{\rm pr}  \cdot (h |_{A[[t]]}) \; : \; A[[t]] \lrto  A[[t]]  \, \mbox{,} $$
 where $\mathop{\rm pr} $ is the projection from the $A$-module
$A((t))= t^{-1} A[t^{-1}] \oplus A[[t]]$ onto the {$A$-submodule} $A[[t]]$.  Here the determinant of the map $s=d_f \cdot d_g  \cdot d_{fg}^{-1}$
from the
$A$-module $A[[t]]$ to the $A$-module $A[[t]]$
  is well-defined, since there is an integer $n > 0$ such that $s |_{t^nA[[t]]} = {\rm id}$.

We can change the commutative ring $A$, and we will obtain the covariant group functor $A \mapsto \AutL(A)$
from the category of commutative rings to the category of groups.
We denote this functor by $\AutL$. Then $2$-cocycles $D$ and $\hat{B}$  are $2$-cocycles on the group functor $\AutL$ with the coefficients in the commutative group functor $\GG_m$, where $\GG_m(A) = A^*$ (see more on cocycles for group functors in Section~\ref{Van Est}).

We will prove in Theorem~\ref{theor-3} that the $2$-cocycles $D^{12}$ and $\hat{B}$ define the equivalent central extensions of the group functor  ${\AutL}_{\Q}$
by the commutative group functor ${\GG_m}_{\Q}$. (Here we restricted the group functors to $\Q$-algebras.)

In particular, as a consequence,  we obtain the corresponding statement for the central extensions of the group $\AutL(A)$ by the group $A^*$ when we fix a $\Q$-algebra $A$.

To prove Theorem~\ref{theor-3} we use Theorem~\ref{Th-2-coc}, where we prove that two central extensions of the group functor ${\AutL}_{\Q}$
by the group functor ${\GG_m}_{\Q}$ are equivalent when the corresponding Lie algebra extensions are equivalent and  the restrictions of the central extensions to the group functor of continuous $A$-algebra automorphisms of $A[[t]]$ are equivalent. We note that Lie algebras and corresponding central extensions of Lie algebras are well-defined, since the group functors are represented by ind-affine ind-schemes, see more on the construction of Lie algebra valued functors (i.e. covariant functors from the category of commutative rings to the category of Lie algebras) in Appendix~A.

In Theorem~\ref{theor-1} we compare the corresponding $2$-cocycles $\Lie D$ and $\Lie \hat{B}$ on Lie algebra valued functors. We prove that
$$12 \Lie D = \Lie \hat{B} \, \mbox{.} $$
To prove this statement we obtain the result that is similar to the statement from the theory of infinite-dimensional Lie groups
that the Lie algebra $2$-cocycle that corresponds to the group Bott-Thurston cocyle is the doubled Gelfand-Fuks cocycle on the Lie algebra of smooth vector fields on the circle. Then we prove that the formal analog of the Gelfand-Fuks cocycle is $6$-fold sum  of the cocycle $\Lie D$.

We  also note that the following interesting question. Do $D^{12}$ and $\hat{B}$ equal as $2$-cocycles? We proved only that their cohomology classes are equal.

In Theorem~\ref{theor-4}  we give a consequence of our results for a part of new formal Riemann-Roch theorem~\footnote{The further developments of these results and their application to the full new formal Riemann-Roch theorem will be given in our subsequent papers.}.
  We consider a separated scheme $S$ over $\Q$.
We consider a ringed space $\Theta$ on $S$, where the structure sheaf of $\Theta$  is locally on $S$ isomorphic to the sheaf of rings $\oo_S((t))$
and the transition automorphisms are continuous.
(Locally on $S$  this ringed space corresponds to the punctured formal neighbourhood of a section of a  smooth morphism of relative dimension~$1$.)
 Then we prove that in the \v{C}ech cohomology
group $\check{H}^2(S, \oo_S^*)$ there is an equality
\begin{equation}  \label{form-intr}
[{\mathcal Det}(\Theta)]^{12}  = \partial \, (c_1(\widetilde{\Omega}^1_{\Theta}) \cup c_1(\widetilde{\Omega}^1_{\Theta}) )  \, \mbox{.}
\end{equation}
Here $[{\mathcal Det}(\Theta)]$ is the class of the determinantal $\oo_S^*$-gerbe on $S$ constructed by $\Theta$ by means of the relative determinants $\det (\cdot \mid \cdot)$.  The sheaf $\widetilde{\Omega}^1_{\Theta}$ is a locally free sheaf of rank $1$ on the ringed space $\Theta$,
and this sheaf is an analog of an invertible sheaf $\omega$ from formula~\eqref{DRR}. We denote the class of this sheaf in ${\rm Pic} \, \Theta$ by $c_1(\widetilde{\Omega}^1_{\Theta})$.
Now we take the $\cup$-product and apply the map $\partial$, which is analog of the direct image $\pi_*$ from formula~\eqref{DBr}. The map $\partial$
is constructed by means of the Contou-Carr\`{e}re symbol $\CC$.

We note that in some sense the constructions which we use for a part of the formal Riemann-Roch theorem on a ringed space
are close to the constructions used by M.~Kapranov and \'{E}.~Vasserot in~\cite{KV}.

We also note the following interesting remark. Suppose that we fixed an affine open cover $S= \bigcup_{i \in I} U_i$ and isomorphisms $\phi_i$ such that the ringed space $\Theta$ restricted to every $U_i$
is isomorphic to $\oo_{U_i}((t))$ via $\phi_i$. Then  we can write the both parts of formula~\eqref{form-intr} as explicit \v{C}ech $2$-cocycles with respect to the cover $\{U_i\}$. Now if $2$-cocycles $D^{12}$ and $\hat{B}$ would equal, then formula~\eqref{form-intr}
would be true also for the corresponding \v{C}ech   $2$-cocycles for the sheaf $\oo_S^*$. This would give the analogy with Deligne-Riemann-Roch theorem, because we would obtain the result not in the cohomology group, but in more refined terms that are the \v{C}ech cocycles in our case.

\subsection{Organization of the paper and acknowledgements}
In Section~\ref{sec-BF} we construct the formal Bott-Thurston cocycle.
In Section~\ref{Autgr} we recall the statements on the group $\AutL(A)$ of continuous automorphisms of the $A$-algebra $A((t))$ and consider the corresponding functor~$\AutL$.
In Section~\ref{sec-loop} we recall the statements about the loop group $L \GG_m$  of the multiplicative group scheme $\GG_m$. By definition,
$L \GG_m(A)= A((t))^*$.
In Section~\ref{Symbol} we recall the definition and basic properties of the Contou-Carr\`{e}re symbol $\CC$.
In Section~\ref{Van Est} we introduce and give some properties of the cohomology groups of a group functor with coefficients in a commutative group functor. These cohomology groups are analogs of Van Est (or continuous) cohomology groups in the theory of topological groups.
In Section~\ref{FB} we introduce the formal Bott-Thurston cocycle $\hat{B}$.

In Section~\ref{Sec-det-ext} we describe the determinantal central extension and the corresponding $2$-cocycle.
In Section~\ref{more} we prove some decompositions for the group functor~$\AutL$.
In Section~\ref{sec-RD} we introduce relative determinants, which are the projective $A$-modules $\det(L \mid M)$ of rank $1$ for some $A$-submodules $L$ and $M$ of $A((t))$.
In Section~\ref{DCE} we introduce the determinantal central extension  $\widetilde{\AutL}$ of $\AutL$ by $\GG_m$, construct a natural (non-group)
section of this central extension and describe explicitly the corresponding $2$-cocycle $D$.

In Section~\ref{Sec-coc-Lie} we consider the corresponding $2$-cocycles $\Lie \hat{B}$ and $\Lie D$ on Lie algebra valued functors.
In Section~\ref{Sec-der} we identify the Lie algebra $\Lie \AutL(A)$ over $A$ with the Lie algebra of continuous $A$-derivations of the $A$-algebra
$A((t))$.
In Section~\ref{sec-theor1} we prove Theorem~\ref{theor-1}, which states  that  $12 \Lie D = \Lie \hat{B}$.

In Section~\ref{Sec-comp} we compare two central extensions given by $2$-cocycles $D$ and $\hat{B}$.
In Section~\ref{Sec-form} we recall the statements on infinitesimal formal groups.
In Section~\ref{Sec-centr} we prove Theorem~\ref{Th-2-coc} on the relation between central extensions of  ${\AutL}_{\Q}$ by ${\GG_m}_{\Q}$
and the corresponding central extensions of Lie algebras.
  In Section~\ref{Sec-main}  we prove Theorem~\ref{theor-3}, which states that $D^{12}= B$ in the group $H^2(\AutL_{\Q}, {\GG_m}_{\Q})$.

In Section~\ref{sec-ringed}  we study the part of formal Riemann-Roch theorem. In Section~\ref{sec-spaces}  we introduce $\oo((t))$-spaces that are the ringed spaces.
In Section~\ref{sec-coc} we prove Theorem~\ref{theor-4} on the part of Riemann-Roch theorem for a $\oo((t))$-space.

In Appendix A we recall some statements on Lie algebra valued functors constructed from groups functors with some conditions (in particularly, represented by  ind-affine ind-schemes) and on the corresponding Lie algebra $2$-cocycles.

 \medskip

I am grateful to A.~N.~Parshin for some comments and providing some references.

\section{Formal Bott-Thurston cocycle}  \label{sec-BF}

In this section we will introduce the formal version of the Bott-Thurston cocycle (or, more exactly, the formal version of the cocycle $\widetilde{B}$ from the Introduction). This formal cocycle will be the formal analog of the right hand side of formula~\eqref{DRR}.

\subsection{Automorphism group of  ring of Laurent series}   \label{Autgr}

Let $A$ be any commutative ring. Let $A((t))= A[[t]][t^{-1}]$ be the ring of Laurent series over $A$.

On the ring $A((t))$ there is the natural topology which makes the group $A((t))$ into a topological group, where the base of neighbourhoods of zero consists of $A$-submodules
$U_n = t^n A[[t]]$, $n \in \dz$.

We consider the group $\Aut_A^{\rm c, alg} (A((t)))$ that consists  of all $A$-auto\-mor\-phisms of the $A$-algebra $A((t))$ that are homeomorphisms of $A((t))$.

Then it is known (see, e.g., an exposition in~\cite[\S~1]{M} and see also \cite{GO2}, where the general    case of algebra of iterated Laurent series is analyzed) that $\Aut_A^{\rm c, alg} (A((t)))$ consists of all  $A$-auto\-mor\-phisms of the $A$-algebra $A((t))$ that are continuous, and
an element $\varphi \in \Aut_A^{\rm c, alg} (A((t)))$ is uniquely determined by the series $\widetilde{\varphi} =\varphi(t)$ in the following way
$$
\varphi(\sum_i a_i t^i) = \sum_i a_i    {\widetilde{\varphi} \, }^i   \, \mbox{,}  \quad a_i \in A   \, \mbox{,}
$$
where the series from the right hand side of the formula  converges  in the topology of $A((t))$.
Moreover, series $\widetilde{\varphi} =\varphi(t) \in A((t))$ are exactly the series of the following kind
\begin{equation}  \label{auto}
a_{-n}t^{-n}  + a_{-n+1} t^{-n +1} + \ldots + a_{-1} t^{-1} +a_0 +a_1t +a_2 t^2 + a_3 t^3 +\ldots  \, \mbox{,}
\end{equation}
where $n \in \dz$, $n \ge 0$, elements $a_{-n}, a_{-n +1}, \ldots, a_{-1}, a_0$  are nilpotent elements from $A$, the element $a_1$ is an invertible element from $A$,
and elements $a_2, a_3, \ldots$  are any elements from~$A$.  (The most difficult here is that an endomorphism $\sum_i a_i t^i \mapsto \sum_i a_i    {\widetilde{\varphi} \, }^i$ is an invertible automorphism of the $A$-algebra $A((t))$ for any $ \widetilde{\varphi}$ of type~\eqref{auto}, see also comment on the proof  of this fact later in~Remark~\ref{way}.)

For any $f =\sum_i b_i t^i  \in A((t))$, where $b_i \in A$, and any $g \in A((t))$ such that $g$ is of type~\eqref{auto} by $f \circ g$ we denote a series from $A((t))$ obtained after the
substitution of the series $g$ into the series $f$ instead of variable $t$, i.e.  $f \circ g= \sum_i b_i g^i$.

We note that for any elements $\varphi_1$ and $\varphi_2$  from  $\Aut_A^{\rm c, alg} (A((t)))$ we have
\begin{equation}  \label{tld}
\widetilde{\varphi_1 \varphi_2}  = \widetilde{\varphi_2}  \circ \widetilde{\varphi_1}  \, \mbox{.}
\end{equation}

We will use also the following notation for $A((t))$:
$$
{\LL} (A) = A((t))  \, \mbox{.}
$$

For brevity, we will call a covariant functor from the category of commutative rings to the category of groups (or Abelian groups) as {\em a group functor} (or {\em a commutative group functor}).

We introduce the  group functor
$$
\AutL \; : \; A \longmapsto \Aut\nolimits_A^{\rm c, alg} (\LL (A))  \,
$$
that on
 homomorphisms of commutative rings $A_1 \to A_2$ is defined  by means of associating  $\widetilde{\varphi}$ with $\varphi$ and  the corresponding map on $\widetilde{\varphi}$  (see also formula~\eqref{auto}).

\subsection{Loop group of $\GG_m$ and Contou-Carr\`{e}re symbol}  \label{Contou}

\subsubsection{Loop group of $\GG_m$}   \label{sec-loop}

Let $\GG_m$ be a group functor that is represented by the multiplicative group scheme $\GG_m$, i.e., $\GG_m(A) =  A^*$ for any commutative ring $A$.

By $L \GG_m$ we denote a group functor defined  as
$$
L \GG_m(A) = \GG_m(A((t))) = A((t))^*  \, \mbox{,}
$$
where $A$ is any commutative ring. The group functor $L \GG_m$ is called the loop group of~$\GG_m$.

We recall now the  more explicit description of $L \GG_m$  (see~\cite[Lemme~(0.7)]{CC2}).

Let $\underline{\dz}$ be a group functor such that $\underline{\dz}(A)$ is the group of locally constant integer-valued functions on $\Spec A$ for any commutative ring $A$.
We have an embedding of $\underline{\dz}(A)$ to $ L \GG_m (A)$ via
$$
\underline{\dz}(A)  \ni \underline{n}   \longmapsto t^{\underline n}  \in L \GG_m (A)  \, \mbox{,}
$$
where $t^{\underline{n}}$  is defined as follows. The element $\underline{n}$ defines a decomposition $A = A_1 \times \ldots \times A_k$
into the finite direct product of rings such that $\underline{n}$ restricted to every $\Spec A_l$ equals a constant function with value $n_l \in \dz$. Then
$t^{\underline{n}} = t^{n_1} \times \ldots \times t^{n_k} \in A((t))^*$.

Let $(L \GG_m)^0$ be a group functor such that for any commutative ring $A$
$$
(L \GG_m)^0(A) =
 \left\{    \sum_{i \in \dz}  b_i t^i  \in A((t)) \, \mid  \,  b_0 \in A^* \; \: \mbox{and} \; \:   b_i \in {\rm Nil}(A) \; \: \mbox{for any} \; \: i < 0    \right\}   \, \mbox{,}
$$
where ${\rm Nil}(A)$ is the nilradical of $A$, i.e. the set of all nilpotent elements of $A$.

Now the natural embedding of the set $ (L \GG_m)^0  (A) $ into the set  $ A((t))$
gives  an embedding of the group $(L \GG_m)^0(A)$ into the group $ L \GG_m (A)$.

Embeddings of $\underline{\dz}$ and $(L \GG_m)^0$ into $L \GG_m$ give the decomposition of the group functor into the direct product of group functors:
$$
L \GG_m = \underline{\dz}  \times (L \GG_m)^0  \, \mbox{.}
$$
This decomposition defines the morphism of group functors:
$$  \nu   \; : \; L \GG_m   \lrto \underline{\dz} \, \mbox{.}
$$

Let $(L \GG_m)^{\sharp}$ be a group functor such that for any commutative ring $A$
$$
(L \GG_m)^{\sharp}(A) =
 \left\{    \sum_{i \in \dz}  b_i t^i  \in A((t)) \, \mid  \,  b_0 -1  \in {\rm Nil}(A) \; \: \mbox{and} \; \:   b_i \in {\rm Nil}(A) \; \: \mbox{for any} \; \: i < 0    \right\}   \, \mbox{.}
$$

It is clear that $(L \GG_m)^{\sharp}(A)   \subset (L \GG_m)^0(A)$. Besides, we have $\GG_m(A)  \subset (L \GG_m)^0(A)$, where we consider an invertible element from $A$ as the series consisting of only a constant term. Thus we obtain
$$
(L \GG_m)^0 = (L \GG_m)^{\sharp}  \cdot \GG_m  \, \mbox{.}
$$
We note that this multiplicative decomposition is not the direct product, because inside $(L \GG_m)^0(A)$ we have
$$(L \GG_m)^{\sharp}(A) \cap \GG_m(A) = 1+ {\rm Nil}(A)  \mbox{.} $$

\subsubsection{Contou-Carr\`{e}re symbol} \label{Symbol}

Now we recall the definition of the Contou-Carr\`{e}re symbol (see~\cite{CC1}, \cite[\S~2.9]{D2}, \cite[\S~2]{OZ}).

Let $A$ be a commutative ring. For any $a \in A^*$ and $\underline{n}  \in \underline{\dz}(A) $  we define an element $a^{\underline{n}}$
in $A^*$ in the following way.  The element $\underline{n}$ defines a decomposition $A = A_1 \times \ldots \times A_k$
into the finite direct product of rings such that $\underline{n}$ restricted to every $\Spec A_l$ equals a constant function with value $n_l \in \dz$. Then
$a^{\underline{n}} = a^{n_1} \times \ldots \times a^{n_k} $.

We consider the following free $A((t))$-module of rang $1$:
\begin{equation}  \label{form-ker}
\widetilde{\Omega}^1_{A((t))}  = \Omega^1_{A((t))}  / N  \, \mbox{,}
\end{equation}
where $\Omega^1_{A((t))}$ is the $A((t))$-module of absolute K\"ahler differentials, the  $A((t))$-submodule $N$ is generated by all elements
 $df -  f' dt$, where $f \in A((t))$ and $f'= \frac{\partial{f}}{\partial{t}}$. We note that $N$ contains elements $da$, where $a \in A$. It is clear that $dt$ is a basis of the $A((t))$-module  $\widetilde{\Omega}^1_{A((t))}$.

Now we define residue
$$
\res \; : \;  \Omega^1_{A((t))}  \lrto A
$$
as the composition of the natural map  $\Omega^1_{A((t))} \to \widetilde{\Omega}^1_{A((t))} $ and the map $\sum_{i \in \dz} a_it^i dt   \mapsto a_{-1}$.

{\em The Contou-Carr\`{e}re symbol} is the  bimultiplicative antisymmetric morphism
$$
{\CC} \; : \; L \GG_m  \times L \GG_m \lrto \GG_m
$$
uniquely defined by the following additional  three properties.
\begin{enumerate}
\item  If $\Q \subset A$ and $f,g  \in L \GG_m(A) = A((t))^*$, then
\begin{equation}  \label{CC-exp-log}
\CC(f,g) = \exp \res \left(\log f \cdot \frac{dg}{g} \right) \quad  \mbox{when}  \quad f \in (L \GG_m)^{\sharp}(A)  \, \mbox{,}
\end{equation}
where  $\exp (x)$ and $\log(1+y)$ are the usual formal series, the series $\log$ in above formula  converges in the topology of $A(((t))$,
and application of series $\exp$ in above formula makes sense, because  $\res \left(\log f \cdot \frac{dg}{g} \right)  \in {\rm Nil }(A)$.
\item If $a \in A^*$, then ${\CC}(a, g ) = a^{\nu(g)}$.
\item  $\CC(t,t)=-1$.
\end{enumerate}
The unique extension of the Contou-Carr\`{e}re symbol $\CC$ to an arbitrary ring $A$ is given as follows. We note that direct calculation with
formula~\eqref{CC-exp-log} gives for elements $1 - a_i t^i$ and $1 - b_i t^i$ from $A((t))^*$ when $i > 0$, $j < 0$:
\begin{equation}  \label{CC-atom}
\CC (1 - a_i t^i, 1 - b_j t^j) = \left(1 - a_i^{-j /(i,j)} b_j^{i/(i,j)} \right)^{(i,j)} \, \mbox{,}
\end{equation}
where $(i, j) > 0$ denotes the greatest common divisor of $i$ and $j$.
Besides,
$${\CC (1 - a_i t^i, 1 - b_j t^j) =1}$$
when $i$ and $j$ have the same sign.
We see that  formula~\eqref{CC-atom} does not use that $\Q \subset A$.
For any $f, g \in A((t))^*$ there are  unique decompositions:
$$
f = \prod_{i < 0} (1 - a_it^i) \cdot a_0 \cdot t^{\nu(f)}  \cdot \prod_{i > 0} (1 - a_it^i) \,
\mbox{,}  \qquad
g = \prod_{j < 0} (1 - b_jt^j) \cdot b_0 \cdot t^{\nu(g)}  \cdot \prod_{j > 0} (1 - b_jt^j)   \, \mbox{,}
$$
where $a_0, b_0 \in A^*$, $a_i, b_j \in {\rm Nil}(A) $ when $i, j <0$, and the products over negative $i$ and over negative $j$ are finite products.
Now formula~\eqref{CC-atom} leads to the following formula:
\begin{equation}  \label{CC-form}
\CC(f,g)= (-1)^{\nu(f)  \nu(g)}  \frac{a_0^{\nu(g)} \prod_{i > 0} \prod_{j>0 }
 \left(1 - a_i^{j /(i,j)} b_{-j}^{i/(i,j)} \right)^{(i,j)}  }{b_0^{\nu(f)}
 \prod_{i > 0} \prod_{j>0 }
 \left(1 - a_{-i}^{j /(i,j)} b_{j}^{i/(i,j)} \right)^{(i,j)}
  }  \, \mbox{,}
\end{equation}
where the products in  the numerator and denominator actually consist of a finite number of factors, therefore the formula makes sense.

For any commutative ring $A$, its Milnor $K_2$-group
$$
K_2^M(A)= A^* \otimes_{\dz} A^* / St  \, \mbox{,}
$$
where the subgroup $St \subset A^* \otimes_{\dz} A^*$ is generated by all elements $a \otimes (1-a)$ with $a, 1-a \in A^*$. (These elements are called the Steinberg relations.)

We introduce the group functor $LK_2^M$ as follows
$$
L K_2^{M}(A) = K_2^M(A((t)))  \, \mbox{,}
$$
where $A$ is any commutative ring.

Then the Contou-Carr\`{e}re symbol factors through the natural morphism \linebreak ${L \GG_m  \times L \GG_m  \to LK_2^M}$:
$$
\CC \; : \; L \GG_m  \times L \GG_m  \lrto LK_2^M  \lrto  \GG_m  \, \mbox{.}
$$

\subsection{Cohomology of group functors and formal version of cocycle}

\subsubsection{Analog of Van Est cohomology of group functors}  \label{Van Est}

We will say that a group functor $G_1$ acts on a commutative group functor $G_2$ if the group $G_1(A)$ acts on the Abelian group $G_2(A)$ for any commutative ring $A$ and these actions are compatible for any homomorphisms of commutative rings $A_1 \to A_2$.

For commutative group functors $F_1$ and $F_2$ by $\Hom^{\rm gr}(F_1, F_2)$ we denote the Abelian group of all morphisms of group functors from $F_1$ to $F_2$,
i.e. morphisms that preserve the group structure.

For a  (covariant) functor $G$ from the category of commutative rings to the category of sets and a commutative group functor $F$ we denote by $\Hom(G,F)$ the Abelian group of all morphisms from functor $G$ to functor $F$. We note that even if $G$ is a group functor, in the group $\Hom(G,F)$
we consider $G$ as a functor from the category of commutative rings to the category of sets.

For a  functor $G$ from the category of commutative rings to the category of sets  and an integer $n \ge 1$  we denote by $G^{\times n}$ the functor that is the $n$-fold direct product of the functor $G$.

Let $G$ be a group functor and $F$ be a commutative group functor such that $G$ acts on~$F$. We consider the cochain complex of Abelian groups
\begin{equation}  \label{cochain-functor}
C^0 (G,F)  \stackrel{\delta_0}{\lrto} C^{1}(G, F) \stackrel{\delta_1}{\lrto} \ldots \lrto C^{k}(G,F) \stackrel{\delta_k}{\lrto} \ldots  \, \mbox{,}
\end{equation}
where $C^0(G, F)= F(\dz)$, $C^k(G, F)= \Hom (G^{\times k}, F)$ when $k \ge 1$, and differentials $\delta_{q}$ for integers $q \ge 0$  are given in the following way:
\begin{multline*}
\mathop{\delta_{q}c} \, (g_1, \ldots, g_{q+1}) = g_1  \mathop{c}  (g_2, \ldots, g_{q+1}) \, \cdot \,  \prod_{i=1}^q  \mathop{c}  (g_1, \ldots, g_i g_{i+1}, \ldots, g_{q+1})^{(-1)^i } \,  \cdot\\
\cdot \,   \mathop{c}  (g_1, \ldots, g_q)^{(-1)^{q+1} }  \, \mbox{,}
\end{multline*}
where $c \in C^q(G,F)$, $g_j \in G(A)$ with $1 \le j \le q+1$ for any commutative ring $A$.

Complex~\eqref{cochain-functor} is well-defined, i.e. $\delta_{q+1} \delta_q =0 $ when $q \ge 0$, since the last equality is true for any commutative ring $A$, where this equality is the part of bar or standard resolution to calculate the group cohomology.

\begin{nt} \em
We used the multiplicative notation for the group laws in  Abelian groups $F(A)$ and $C^{q}(G,F)$, since  in this paper we  have examples of these groups only
with the multiplicative notation.
\end{nt}

\begin{defin}
Suppose that a group functor  $G$ acts on a commutative group functor $F$. We consider complex~\eqref{cochain-functor}.

By a $q$-cocycle on $G$ with coefficients in $F$ we call an element of the subgroup  \linebreak $\Ker \delta_q \subset C^q(G,F)$, where $q \ge 0$.

By a $q$-coboundary on $G$ with coefficients in $F$ we call an element of the subgroup $\Image \delta_{q-1} \subset C^q(G,F)$, where $q \ge 1$.

We introduce the Abelian group $H^q(G,F)= \Ker \delta_q / \Image \delta_{q-1} $ for any $q \ge 0$.
\end{defin}

\begin{nt} {\em
The definition of cohomology groups $H^q(G,F)$ is analogous to the definition of Van Est (or continuous) cohomology groups for topological groups and topological modules, see, e.g.,~\cite[ch.~1, \S~1.2.B]{FF},~\cite{Fu}.
If the functors $G$ and $F$ are represented by a group scheme or a group ind-scheme, then this analogy becomes especially clear.
}
\end{nt}

\begin{nt} \label{over_R}  {\em
We fix a commutative ring $R$.
If under conditions of the definition,  $G$ and $F$ are functors over $R$, i.e. functors from the category of commutative $R$-algebras, then in complex~\eqref{cochain-functor} one has to replace $C^0(G,F)= F(\dz)$ to $C^0(G,F)= F(R)$.
}
\end{nt}

We immediately obtain the following proposition.

\begin{prop}  \label{coh-gr-fun}
Suppose that a group functor  $G$ acts on a commutative group functor~$F$.
\begin{enumerate}
\item
 A collection $\{ c_A \}$ of $q$-cocycles on  $G(A)$ with coefficients in  $F(A)$, where $A$ runs over all commutative rings, together with obvious compatibility condition between $q$-cocycles $c_{A_1}$ and $c_{A_2}$ for any homomorphism of commutative rings $A_1 \to A_2$ defines the $q$-cocycle on $G$ with coefficients in $F$.
\item
For any commutative ring $A$ there is a natural morphism  from complex~\eqref{cochain-functor}
to the complex obtained from the  bar resolution and that calculates the  cohomology of the group $G(A)$ with coefficients in the Abelian group $F(A)$.
Hence we obtain a map from $q$-cocycles on $G$ with coefficients in $F$ (or $q$-coboundaries on $G$ with coefficients in~$F$)
to $q$-cocycles on $G(A)$ with coefficients in  $F(A)$ (or $q$-coboundaries on  $G(A)$ with coefficients in  $F(A)$), and a natural map
  $$
H^q(G, F) \lrto H^q(G(A), F(A))  \, \mbox{.}
$$
 \end{enumerate}
\end{prop}

As usual, there is the relation with extensions of group functors that is given in the following proposition.
\begin{prop}  \label{Ext-funct}
Suppose that a group functor  $G$ acts on a commutative group functor $F$. Then elements of the group $H^2(G, F)$
are in one-to-one correspondence with equivalence classes of  extensions of the group functor $G$ by the  group functor $F$
\begin{equation}  \label{ext}
1 \lrto F \lrto \widetilde{G} \stackrel{\pi}{\lrto} G \lrto 1
\end{equation}
such that the morphism $\pi$ has a section $G \to \widetilde{G}$ that is a morphism of functors (in general,   not group functors, i.e. $G(A) \to \widetilde{G}(A)$ is a map only of the sets for any commutative ring $A$), and
the action of $G$ on $F$ comes from inner automorphisms in the group functor $\widetilde{G}$.
\end{prop}
\begin{proof}
This is standard. Choose a section $\sigma : G \to \widetilde{G}$. Then we define the $2$-cocycle $\Lambda$ on $G$ with coefficient in $F$ as follows:
$$
\sigma(g_1) \sigma(g_2) = \Lambda(g_1, g_2) \sigma(g_1 g_2)  \, \mbox{,}
$$
where $g_1, g_2 \in G(A)$ for any commutative ring $A$. When we change the section $\sigma$, then the $2$-cocycle $\Lambda$ will change  by a $2$-coboundary.
\end{proof}

\begin{nt} \em
It is important that the morphism $\pi$ in a group functor extension~\eqref{ext} has a section  $G \to \widetilde{G}$ from $\Hom (G, \widetilde{G})$.
\end{nt}

\begin{nt}   \label{central} \em
A group functor extension~\eqref{ext} from Proposition~\ref{Ext-funct} is central, i.e. the subgroup $F(A)$ is in the center of the group $\widetilde{G}(A)$ for any commutative ring $A$, if  and only if $G$ acts trivially on $F$.
\end{nt}

For any two commutative group functors $F_1$ and $F_2$ we denote by $F_1 \otimes F_2$ the commutative group functor defined as
$(F_1 \otimes F_2) (A) = F_1 (A) \otimes_{\dz} F_2(A)$ for any commutative ring $A$.

\begin{nt}   \label{cup-funct}
\em  Clearly, by Proposition~\ref{coh-gr-fun}, the $\cup$-products between cohomology groups of a group functor $G $
with coefficients in a commutative group functor $F$ are written in the same way as $\cup$-products in group cohomology by means of the same explicit formulas on cocycles, see Section~\ref{prod-gr-coh}.

\end{nt}

\begin{nt}  \label{hom-cocycle}
\em Let $F_1$ and $F_2$ be two commutative group functors with the action of the group functor $G$ on them. Clearly, any morphism of group functors
$F_1 \to F_2$ (i.e. an element from $\Hom^{\rm gr}(F_1, F_2)$) that commutes with the action of $G$ induces the morphism of the corresponding
complexes~\eqref{cochain-functor}, and hence the homomorphism of corresponding cocycles, coboundaries and cohomology groups.
\end{nt}

\subsubsection{Formal version of cocycle} \label{FB}

For any commutative ring $A$  the action of the group $\AutL(A)  = \Aut\nolimits_A^{\rm c, alg} (\LL (A)) $ on the $A$-algebra $\LL(A)=A((t))$
 induces the action of the group $\AutL(A)$ on the Abelian group $L \GG_m(A) = A((t))^*$.

This gives the action
of the group functor $\AutL$  on the commutative group functor $L\GG_m $.

Let $A$ be any commutative ring. Then $\LL(A)$-module  $\widetilde{\Omega}^1_{\LL(A)}$ contains an element
 $d t $, which is the basis of this module. For any   element  $\varphi \in \AutL(A)$
we define (see also notation in Section~\ref{Autgr})
\begin{equation}  \label{form-tau}
\tau(\varphi)  = \frac{d \varphi(t)}{   dt} =  = {\widetilde{\varphi}}'  \in \LL(A)^*=A((t))^*  \, \mbox{.}
\end{equation}
For any elements $\varphi_1$ and $\varphi_2$ from  $\AutL(A)$  we have
$$
\tau({\varphi_1 \varphi_2})= {\widetilde{ \varphi_1 \varphi_2}}'= (\widetilde{\varphi_2} \circ \widetilde{\varphi_1})' =
 ({\widetilde{\varphi_2}}' \circ \widetilde{\varphi_1}) \cdot {\widetilde{\varphi_1}}' = \tau(\varphi_1) \cdot \varphi_1(\tau(\varphi_2))  \, \mbox{.}
$$

Hence we obtain that $\tau$ is a $1$-cocycle on  the group functor  $\AutL$ with coefficients in the commutative group functor $L \GG_m$.

The image of $\tau \times \tau$  under the $\cup$-product  (see Section~\ref{prod-gr-coh} and Remark~\ref{cup-funct}) is the $2$-cocycle
$\tau \cup \tau$
 on  $\AutL$
with coefficients in $L \GG_m  \otimes L \GG_m$.

We note that
the Contou-Carr\`{e}re symbol
$$\CC \;  : \;  L \GG_m  \otimes L \GG_m  \lrto \GG_m$$  commutes with the diagonal action of the group functor
$\AutL$ on the commutative group functor $L \GG_m  \otimes L \GG_m$, see, e.g.,~\cite{GO3}, where this property is proved in more general situation for the $n$-dimensional Contou-Carr\`{e}re symbol that generalizes the (one-dimensional) Contou-Carr\`{e}re symbol $\CC$.

Therefore after the apllication of the Contou-Carr\`{e}re symbol $\CC$ to the $2$-cocycle $\tau \cup \tau$ we obtain the well-defined
$2$-cocycle on $\AutL$
with coefficients in $ \GG_m$, see Remark~\ref{hom-cocycle}.

Thus, we obtained the following proposition.

\begin{prop}[Formal Bott-Thurston cocycle]  \label{FormCoc}
An explicit formula
\begin{equation}  \label{FBT}
\hat{B}(\varphi_1,  \varphi_2)=
\CC(\tau(\varphi_1), \, \varphi_1(\tau(\varphi_2)))=
\CC (\widetilde{\varphi_1}', \,  \widetilde{\varphi_2}'\circ \widetilde{\varphi_1}) \, \mbox{,}
\end{equation}
where elements $\varphi_1 $ and $\varphi_2$ are from the group $\AutL(A)$, and $A$ is any commutative ring,
gives a well-defined $2$-cocycle $\hat{B}$, which we call the formal Bott-Thurston cocycle,  on the group functor $\AutL$
with coefficients in the group functor
$ \GG_m$.
\end{prop}

\begin{nt}  \em
It is possible to rewrite formula~\eqref{FBT} in the spirit of formula~\eqref{Bo}:
$$
\hat{B}(\varphi_1,  \varphi_2)= \CC (\widetilde{\varphi_1}', \,  (\widetilde{\varphi_2} \circ \widetilde{\varphi_1})') =
\CC(\tau(\varphi_1), \tau(\varphi_1 \varphi_2))  \, \mbox{,}
$$
where the first equality follows from the bimultiplicative property of $\CC$ and the fact  $\CC(\widetilde{\varphi_1}', \, \widetilde{\varphi_1}') =1$
that follows, for example, from formula~\eqref{CC-form}.
\end{nt}

\begin{nt}  \em
For the construction of the $2$-cocycle $\hat{B}$
 it was important that the Contou-Carr\`ere symbol  $\CC$ is invariant under the diagonal action of the group functor $\AutL$
 on the commutative group functor $L \GG_m  \otimes L \GG_m$.

 The Contou-Carr\`ere symbol $\CC$ factors through the commutative group functor $L K_2^M$
 (see the end of Section~\ref{Symbol}). Besides, the diagonal action of  $\AutL$
 on  $L \GG_m  \otimes L \GG_m$ induces the action of  $\AutL$ on $L K_2^M$.

 Now from~\cite[Theorem~8.10]{GO1}, \cite{GO4}, \cite[\S~7]{GO5} it follows that {\em any} morphism of group functors from   $L K_2^M$  to $\GG_m$, i.e. an element
  from $\Hom^{\rm gr}(L K_2^M ,\GG_m)$,
  is induced by $(\CC)^i$, where $i \in \dz$ and $(\CC)^i$ means the composition of the Contou-Carr\`ere symbol $\CC$ and the morphism
 $$\GG_m \lrto \GG_m  \;  :  \;  g \longmapsto g^i  \, \mbox{,} \quad g \in \GG_m(A)=A^* \, \mbox{,} \quad  A \quad \mbox{is any commutative ring}   \, \mbox{.}$$

 Hence we obtain that {\em any} morphism of group functors from $L K_2^M$  to $\GG_m$ commutes with the action of $\AutL$ on $L K_2^M$,
 and therefore this morphism can be also used to construct a $2$-cocycle  on $\AutL$ with coefficients in $\GG_m$. In this case we obtain the $i$-th multiple of the
 $2$-cocycle $\hat{B}$ for some $i \in \dz$.
\end{nt}

\begin{nt}  \em
The maps which we considered above for the construction of the $2$-cocycle~$\hat{B}$ induce the following maps on cohomology groups
\begin{multline*}
H^1(\AutL, \GG_m) \otimes_{\dz} H^1(\AutL, \GG_m)  \lrto H^2(\AutL, L \GG_m \otimes L \GG_m )  \lrto  \\
\lrto
H^2(\AutL, L K_2^M )  \lrto H^2(\AutL, \GG_m)  \, \mbox{,}
\end{multline*}
where the first map is the $\cup$-product, and the composition of the last two maps is induced by the Contou-Carr\`ere symbol $\CC$.
\end{nt}

\begin{nt}  \label{Nilp}  \em
Suppose that ${\rm Nil}(A) = 0$ for a commutative ring $A$.
The from formula~\eqref{FBT} and formula~\eqref{CC-form}  for the Contou-Carr\`ere symbol $\CC$  it follows that $\hat{B}(\varphi_1,  \varphi_2) =1 $
for any $\varphi_1 $ and $\varphi_2$  from  $\AutL(A)$.
\end{nt}

\begin{nt} \label{higher} \em
There is the $n$-dimensional Contou-Carr\`ere symbol $\CC_n$ that is
a multilinear antisymmetric functorial map of $n+1$ variables:
$$
A((t_1))\ldots ((t_n))^*  \, \times \ldots \times  \, A((t_1))  \ldots ((t_n))^*  \lrto A^*  \, \mbox{,}
$$
where  $A$ is a  commutative ring, and
$\CC_1 = \CC$.
When $A$ is a $\Q$-algebra, then the main ingredient for $\CC_n$ is the formula
$$
\CC\nolimits_n (f_1, \ldots , f_{n+1}) = \exp \res \left( \log f_1 \cdot \frac{d f_2}{f_2}  \wedge \ldots \wedge \frac{d f_{n+1}}{f_{n+1}} \right) \, \mbox{,}
$$
where $\res$ is the $n$-dimensional residue that is equal to the coefficient by \linebreak
${t_1^{-1} \ldots t_n^{-1} dt_1 \wedge \ldots dt_n}$, elements $f_1, \ldots, f_{n+1}  \in A((t_1)) \ldots ((t_{n}))^*$, and there is some condition on $f_1$
so that the series $\log f_1$ converges in the natural topology of $A((t_1))  \ldots ((t_{n}))$, see~\cite{OZ}, where it was done  for $n=2$, and~\cite{GO1, GO5} for any $n$.

Then it is possible to write  the  formal version of the slightly changed $(n+1)$-cocycle $B$ from formula~\eqref{Bo} for the $n$-dimensional torus using the $n$-dimensional Contou-Carrere symbol $\CC_n$, the formal analog of formula~\eqref{Bo} and the differential form $\omega = dt_1 \wedge \ldots \wedge dt_n $.
 This $(n+1)$-cocycle is on the group of continuous $A$-automorphisms of the $A$-algebra $A((t_1))  \ldots ((t_n))$ with values  in the group $A^*$, where $A$ is any commutative ring. When $n=1$ this cocycle coincides with  the cocycle $\hat{B}$ from Proposition~\ref{FormCoc}.

The study of this $(n+1)$-cocycle  coming from $A((t_1))  \ldots ((t_n))$ will be the subject of our futher investigation.
\end{nt}

\section{Cocycle from the determinantal central extension}   \label{Sec-det-ext}
In this section we describe the determinantal central extension of $\AutL$  by $\GG_m$, an explicit section of this central extension as functors but not as group functors, and a $2$-cocycle on $\AutL$  with values in $\GG_m$ obtained from the central extension and the section.

\subsection{More on the automorphism group of ring of Laurent series}  \label{more}

We define now  the following  subfunctors of the  functor  $\AutL$:
\begin{equation}  \label{Subfunct}
\AutLp \, \mbox{,}  \quad  \AutLm  \, \mbox{,} \quad \AutLpo \, \mbox{,} \quad  \AutLmo \mbox{.}
\end{equation}

Let $A$ be any commutative ring.

We define the subset $\AutLpo (A)   $  of the group $\AutL (A)$
that consists of elements $\varphi$
such that the corresponding elements $\tilde{\varphi} = \varphi(t) \in A((t))^*$
are exactly the elements of the kind
\begin{equation}  \label{f1}
a_1 t + a_2 t^2 + \ldots \, \mbox{,}
\end{equation}
where $a_1 \in A^*$, and elements $a_i$ with $i \ge 2$ are any from $A$.
It is easy to see that $\AutLpo (A)   $ is a subgroup.

We define the subset $\AutLp (A)   $
of the group $\AutL (A)$
that consists of elements $\varphi$
such that the corresponding elements $\tilde{\varphi} = \varphi(t) $
are exactly the elements of the kind
\begin{equation}  \label{f2}
a_0 + a_1t + a_2 t^2 + \ldots   \, \mbox{,}
\end{equation}
where $a_0$ is from the nil-radical $ {\rm Nil} (A)$, $a_1 \in A^*$, and elements $a_i$ with $i \ge 2$ are any from $A$.

We define the subset $\AutLmo(A)$ of the group $\AutL(A)$ that consists of elements $\varphi$ such that the corresponding
elements  $\tilde{\varphi} = \varphi(t) $ are exactly the elements of the kind
\begin{equation}   \label{f3}
a_{-n} t^{-n} + a_{-n +1} t^{-n +1} + \ldots + a_{-1}t^{-1} + a_0 + t   \, \mbox{,}
\end{equation}
where $n \in \dz$, $n \ge 0$, all elements $a_i$ are from ${\rm Nil}(A)$.

We define the subset $\AutLm(A)$ of the group $\AutL(A)$ that consists of elements $\varphi$ such that the corresponding
elements  $\tilde{\varphi} = \varphi(t) $ are exactly the elements of the kind
\begin{equation}  \label{f4}
a_{-n} t^{-n} + a_{-n +1} t^{-n +1} + \ldots + a_{-1}t^{-1} + t   \, \mbox{,}
\end{equation}
where $n \in \dz$, $n \ge 1$, all elements $a_i$ are from ${\rm Nil}(A)$.

Thus we have the following embeddings of sets:
\begin{gather*}
\AutLpo(A) \subset \AutLp(A)  \subset \AutL(A) \\ \AutL(A) \supset  \AutLmo(A) \supset \AutLm(A)  \, \mbox{.}
\end{gather*}

\begin{prop}  \label{Decomp}
Let $A$ be any commutative ring. There are the following unique decompositions
\begin{multline}  \label{dec-sub}
\AutL(A) = \AutLpo(A) \cdot \AutLmo(A)  =  \AutLp(A)  \cdot  \AutLm(A) = \\
= \AutLmo(A) \cdot \AutLpo(A) = \AutLm(A) \cdot \AutLp(A)
\, \mbox{.}
\end{multline}
All subsets used in these decompositions are subgroups, i.e. subfunctors~\eqref{Subfunct} are group subfunctors.  Besides,
these decompositions are functorial with respect to $A$, and therefore give the decompositions of the group functor $\AutL$.
\end{prop}
\begin{nt} \em
Decompositions~\eqref{dec-sub}   are not the direct products and not the semidirect products of groups for some commutative rings $A$.
\end{nt}
\begin{proof}
We note that all subsets used in decompositions~\eqref{dec-sub}  are closed with respect to the multiplication induced from the group structure
of $\AutL(A)$.

We prove the existence of decompositions. First, we prove the existence of the first decomposition.

For the proof we will use the following statement. Let $J$ be an ideal in the ring $A$. Suppose that an element $p =\sum_{i \ge 1} d_i t^i$,
where $d_i \in A$, $d_1 \in A^*$, belongs to $t + J((t))$. Then for the element $q= \sum_{i\ge 1} e_i t^i$, where $e_1 \in A^*$, such that $p \circ q = t$ we have
that $q $ belongs to $t + J((t))$. The proof of the statement follows at once if we consider the images of $p$ and $q$ to the ring
$(A/J)((t))$.

Let $\varphi \in \AutL(A)$ such that
$$
\widetilde{\varphi} = \varphi(t) = \sum_i a_i t^i   \, \mbox{,} \quad \mbox{where} \quad a_1 \in A^* \quad \mbox{and} \quad a_i \in {\rm Nil}(A) \quad \mbox{if} \quad i < 1 \,
\mbox{.}
$$

We denote $g_1 = \widetilde{\varphi}$. Since  $a_1 \in A^*$, there is the element
${h_1 = \sum_{i \ge 1} c_i t^i}$,
where ${c_1 \in A^*}$,
such that $(\sum_{i \ge 1} a_i t^i) \circ h_1 = t$. Hence we obtain that the element \linebreak
${g_2 = g_1 \circ h_1 = \sum_i b_i t^i }$
belongs to $t + I((t)) \subset A((t))$, where $I$ is an ideal in the ring $A$
 generated by all elements $a_i$ with $i \le 0$. The ideal $I$ is generated by the finite number of nilpotent elements. Therefore there is
 an integer $k > 0$ such that $I^k=0$.

We consider  the element  $h_2 = \sum_{i \ge 1} f_i t^i $, where $f_1 \in A^*$,    such that
$
(\sum_{i \ge 1} b_i t^i ) \circ h_2 =t
$. Then by the statement above, $h_2 \in t + I((t))$. Therefore we obtain the element \linebreak
$g_3 = g_2 \circ h_2= \sum_i c_i t^i  \in t + I((t)) $, where  $c_1 \in 1 + I^2$ and $c_i \in I^2$ when $i \ge 2$.

We consider the element $h_3 = \sum_{i \ge 1 } v_i t^i$, where $v_1 \in A^*$, such that
$(\sum_{i \ge 1}  c_i t^i) \circ h_3 = t$.

By iterating this process, we obtain
$g_{k+1} = g_1 \circ h_1 \circ h_2 \ldots \circ h_{k} = \sum_i w_i t^i  \in t + I((t))$, where $w_i  \in I^k =0$ when $i \ge 2$, and
$w_1  \in 1 + I^k$, i.e. $w_1 =1$.

Hence we obtained that $g_1 \circ ( h_1 \circ h_2 \ldots \circ h_{k} )= \sum_{i \le 0} z_i t^i + t$ with $z_i \in {\rm Nil}(A)$. Since we know that
the element  $h_1 \circ h_2 \ldots \circ h_{k}$ corresponds to the element from
  $\AutLpo (A)   $, and  $\AutLpo (A)   $ is a group, we obtain now the first decomposition from~\eqref{dec-sub}.

The second decomposition from~\eqref{dec-sub} follows from the first decomposition, since we always have that an element
$
(\sum_{i \le 0 } c_i t^i + t) \circ (t - c_0)
$
corresponds to the element from $\AutLm(A)$ (see formula~\eqref{f4}).

Now we prove  that   $\AutLmo(A)$ is a group. It is enough to prove that for any element $g = a_{-l}t^{-l} + \ldots + a_{-1}t^{-1} + a_0 + t  $,
where $l \ge 0$ and  $a_{i} \in {\rm Nil}(A)$ when $i \le 0$, there is an element $h =b_{-m} t^{-m} + \ldots + b_{-1} t^{-1} + b_0 + t$, where
$m \ge 0$ and $b_{i} \in {\rm Nil}(A)$ when $i \le 0$, such that $g \circ h = t$. Let $I$ be a nilpotent ideal in the ring $A$
generated by elements $a_{-l}, \ldots, a_{-1}, a_0$.  We note that
$
g \circ (t - a_0 -a_1 t^{-1} - \ldots - a_{-l}t^{-l})  \in t + I^{2}[t^{-1}]
$.
Iterating this process and using the nilpotence of $I$ we obtain the element $h$.

In the same way we can prove that   $\AutLm(A)$ is a group.

To see that $\AutLp(A)$ is a group we note that for any element $\sum_{i \ge 0} d_i t^i$, where $d_0 \in {\rm Nil}(A)$, $d_1 \in A^*$,
we have  $(\sum_{i \ge 0} d_i t^i) \circ g = d_0 + t$, where $(\sum_{i  \ge 1} d_i t^i) \circ g = t$, and, besides, $(d_0 +t) \circ (-d_0 + t) =t$.

Now the third and  fourth decompositions in~\eqref{dec-sub} follow from the first and second decompositions, because subsets in these decompositions are subgroups and we can apply the map $g \mapsto g^{-1}$.

The uniqueness of the decompositions follow from the facts that the subsets in the decompositions are subgroups and inside $\AutL(A)$ we have (see formulas~\eqref{f1}-\eqref{f4})
$$
\AutLpo(A)  \cap \AutLmo(A) = \AutLp(A) \cap \AutLm(A)  = e  \, \mbox{,}
$$
where $e$ is the identity element  of the group $\AutL(A)$.
\end{proof}

\begin{nt} \em  \label{way}
The proof of Proposition~\ref{Decomp} gives the way how to construct the element $h$ given by formula~\eqref{auto}  for an element $g$  of the same type such that
$g \circ h = h \circ g = t$. (We don't use in the proof the fact that it is possible to construct such an element.)
Besides, from the third decomposition in~\eqref{dec-sub} it is easy to see that for a series $g$ of form~\eqref{auto} the map
$A((t)) \to A((t)) : f \mapsto f \circ g$
is a continuous map.
\end{nt}

\begin{nt} \em
 The group $\AutLp (A)   $ is exactly  the group of all continuous $A$-automorphisms of the {$A$-algebra} $A[[t]]$.
\end{nt}

From the explicit form of elements $\widetilde{\varphi} = \varphi(t)$ for $\varphi \in \AutL (A)$, see formula~\eqref{auto},
it is easy to see that the  functor  $\AutLp$ is represented by the following ind-scheme
\begin{equation}  \label{ind-sch}
\Spec \dz[a_1, a_1^{-1}] \times \Spec \dz[a_2, a_3, a_4, \ldots]  \times  \mbox{``$\varinjlim\limits_{\{\epsilon_i\}}$''} \Spec \dz [a_0, a_{-1}, a_{-2}, \ldots]/ I_{\{\epsilon_i\}}  \, \mbox{,}
\end{equation}
where  the limit is taken over all the sequences $\{\epsilon_i\}$ with integers $i \le 0$ and $\epsilon_i$ are nonnegative integers such that all but finitely many $\epsilon_i$ equal zero, the ideal $I_{\{\epsilon_i\}}$ is generated by elements $a_i^{\epsilon_i +1}$ for all integers $i \le 0$.

We note that decomposition in formula~\eqref{ind-sch} is not the decomposition of group ind-schemes (when we consider elements of type~\eqref{auto}).
But it is easy to see that functors~\eqref{Subfunct} are also represented by ind-schemes. Therefore decompositions in Proposition~\ref{Decomp} give decompositions of group ind-schemes (but these decompositions are not direct products and not semidirect products.)

\subsection{Relative determinants}  \label{sec-RD}

The following proposition follows from the general theory of Tate $A$-modules, see~\cite[\S~2.13]{BBE}. We give here a simple proof for the case of $A((t))$.
\begin{prop} \label{cont}
Let $A$ be a commutative ring.
\begin{enumerate}
\item For any $A$-module homeomorphism (i.e. a homeomorphism that preserves the $A$-module structures)  $g: A((t))  \to A((t))$   and any $n \in \dz$ there are $l,m \in \dz$ such that
$$
t^m A[[t]]  \subset g (t^n A[[t]])  \subset t^l A[[t]]  \, \mbox{.}
$$
\item For any  $A$-module homeomorphisms  $g$ and  $h$ from $A((t))$ to $A((t))$  and any $m,n \in \dz$ such that $g (t^n A[[t]]) \supset h (t^m A[[t]])$ we have that
$g (t^n A[[t]]) / h (t^m A[[t]])$ is a finitely generated projective $A$-module.
\end{enumerate}
\end{prop}
\begin{proof}
1. This follows from continuous property of $g$ and $g^{-1}$, and that an $A$-module $t^rA[[t]]/ t^s A[[t]]$ is finitely generated.

2. By applying  the map $h^{-1}$, we obtain that it is enough to prove this item for the identity map $h$. Now we have that $A((t))/ t^m A[[t]]$
is a free $A$-module. The module
$$A((t))/ g (t^n A[[t]]) = g A((t))/g (t^n A[[t]]) \simeq A((t))/ t^n A[[t]] $$ is also a free $A$-module. Therefore from the exact sequence of $A$-modules
$$
0 \lrto g (t^n A[[t]]) / t^m A[[t]]  \lrto A((t))/ t^m A[[t]]  \lrto A((t))/ g (t^n A[[t]])  \lrto 0
$$
we obtain that $A((t))/ t^m A[[t]]  \simeq g (t^n A[[t]]) / t^m A[[t]]  \oplus A((t))/ g (t^n A[[t]])$. This implies that $g (t^n A[[t]]) / t^m A[[t]]$ is a projective $A$-module. Now we will prove  that this module is finitely generated.  By item 1, there is $l \in \dz$ such that
$g (t^n A[[t]]) \subset t^l A[[t]]$, and, as we have just proved, $t^l A[[t]]  / g (t^n A[[t]])$ is a projective $A$-module. Therefore we have that
$$t^l A[[t]] /t^m A[[t]]  \, \simeq  \,  g (t^n A[[t]]) / t^m A[[t]]  \, \oplus \, t^l A[[t]]  / g (t^n A[[t]]) \mbox{.}$$ Now we use that the $A$-module
$t^l A[[t]] /t^m A[[t]]$ is finitely generated.
\end{proof}

\bigskip

Let $A$ be a commutative ring.
Let $L$ and $M$ be  $A$-submodules of $A((t))$ that are equal to $g (t^n A[[t]])$ and $h (t^m A[[t]])$ for some
 $A$-module homeomorphisms  $g$ and  $h$ from $A((t))$ to $A((t))$. Let $l$ be an integer
such that $L$ and $M$ contain $t^l A[[t]]$. Then the projective $A$-module of rank $1$
$$
\Hom\nolimits_A \left(\bigwedge^{\rm max}(L/ t^l A[[t]]),  \, \bigwedge^{\rm max}(M/ t^lA[[t]]) \right)
$$
does not depend on the choice of $l$ up to a unique isomorphism. We identify over all such $l$ all these projective $A$-modules via the following
definition of the projective $A$-module of rank $1$
$$
\det(L \mid M) = \varinjlim_l \Hom\nolimits_A \left(\bigwedge^{\rm max}(L/ t^l A[[t]]), \, \bigwedge^{\rm max}(M/ t^lA[[t]]) \right)  \, \mbox{.}
$$

For any $A$-submodule $N$ of the same form as $L$ and $M$ above we have a canonical isomorphism of $A$-modules
$$
\det(L \mid M)  \otimes_A \det(M \mid N) \lrto \det(L \mid N)
$$
that satisfies the associativity diagram for any four $A$-submodules of $A((t))$ of the same form as $L$ and $M$ above. Besides,
any $A$-module homeomorphism  $f: A((t))  \to A((t))$  defines an isomorphism of $A$-modules:
$\det(L \mid M) \to \det(f(L) \mid f(M))$.

\bigskip

For any continuous  $A$-module homomorphism   $g$  from $A((t))$ to $A((t))$ we write $g$ as the block matrix with respect to the decomposition
$A((t)) = t^{-1} A[t^{-1}] \oplus A[[t]]$:
\begin{equation}  \label{matr}
\begin{pmatrix}
a_g & b_g  \\
c_g & d_g
\end{pmatrix}
\end{equation}
where $d_g : A[[t]]  \to A[[t]]$, $d_g= \mathop{\rm pr} \cdot (g |_{A[[t]]})$ and $\mathop{\rm pr}: A((t))  \to A[[t]]$
is the projection. (Here  the matrix acts from the left on an element-column from $A((t))$.) We note that $d_g$ is continuous as the composition of continuous maps.

\begin{prop}  \label{d}
For any  $A$-module homeomorphism  $g$  from $A((t))$ to $A((t))$ such that $d_g$ is bijective,
$\det( A[[t]] \mid g (A[[t]]) )$ is a free $A$-module of rank $1$.
\end{prop}
\begin{proof}
Since $g$ is continuous, there is $L =t^n A[[t]]$ with  $n \ge 0$  such that $g(L)  \subset A[[t]] $. Hence we have that $d_g |_L = g |_L$. Hence we obtain the following isomorphism:
$$
\frac{A[[t]]}{g(L)}= \frac{A[[t]]}{d_g(L)}  \stackrel{d_g^{-1}}{\longrightarrow}  \frac{A[[t]]}{L}
\stackrel{g}{\longrightarrow} \frac{g(A[[t]])}{g(L)}
$$
Since all the above $A$-modules are projective, taking some $N =t^k A[[t]] \subset g(L)$, we obtain the isomorphism of $A$-modules
$A[[t]]/N  \to g(A[[t]])/ N $, which gives the isomorphism of the corresponding top exterior powers of these projective $A$-modules.

\end{proof}

\subsection{Determinantal central extension}  \label{DCE}
\begin{prop}
Let $A$ be a commutative ring. For any $g \in \AutL(A)$ and any $n \in \dz$ we have
$$
A((t))= t^{n-1}A[t^{-1}]  \oplus g(t^n A[[t]])  \, \mbox{.}
$$
\end{prop}
\begin{proof}
By Proposition~\ref{Decomp}, let $g = g_1 g_2$, where $g_1 \in \AutLmo(A)$, $g_2 \in \AutLpo(A)$. It is easy to see that
$g_2(t^n A[[t]]) = t^n A[[t]]$. Therefore we can suppose that ${g = g_1}$.
 Then from the fact that the set ${\mathcal Aut}_{-,1}^{\rm c, alg} ({\mathcal L} )(A)$
is a group it is easy to see that  $ g_1^{-1} (t^{n-1}A[t^{-1}] ) = t^{n-1}A[t^{-1}]  $. Now let us act by the element $g_1^{-1}$ on both sides of the equality from Proposition 3.4, which is to be proved. We get an obvious equality. So the original equality is also true.
\end{proof}

Taking $n=0$, we obtain at once  the following corollary (see notation in~\eqref{matr}).
\begin{cons}  \label{Con}
For any $g \in \AutL(A)$ the map $d_g$ is bijective.
\end{cons}

We define the central extension
\begin{equation}  \label{det-cent-ext}
1 \lrto A^*  \lrto \widetilde{\AutL}(A)  \lrto \AutL(A)  \lrto  1  \, \mbox{,}
\end{equation}
where the group $\widetilde{\AutL(A)}$ consists of pairs $(g, s)$, where $g \in \AutL(A)$
and $s$ is an element of projective $A$-module  $\det(g(A[[t]])  \mid  A[[t]])$ of rank $1$ such that for any prime ideal  $P$ of the ring  $A$ we have   $s \notin P \det(g(A[[t]])  \mid  A[[t]])$.
 The group law is as follows
$$
(g_1, s_1) (g_2, s_2)= (g_1 g_2,  g_1(s_2)  \otimes s_1)  \, \mbox{.}
$$
By Proposition~\ref{d} and Corollary~\ref{Con} this central extension is well-defined.

Clearly, $\widetilde{\AutL} : A \mapsto  \widetilde{\AutL}(A)$ is a group functor that gives the central extension of group functor $\AutL$
by group functor $\GG_m$, and we call this central extension as {\em the determinantal central extension}.
(The functoriality of ${A \mapsto  \widetilde{\AutL}(A)}$ follows from the natural isomorphism
$$(g (A_1[[t]]) / h (A_1[[t]])) \otimes_{A_1} A_2   \xrightarrow{\sim}   u(g) (A_2[[t]]) / u(h) (A_2[[t]])  \mbox{,}$$
which is satisfied, since the projective $ A_1$-module $g (A_1[[t]]) / h (A_1[[t]]))$ is a direct summand of an $A_1$-module $t^lA_1[[t]]/ t^mA_1[[t]]$ for appropriate $l, m \in \dz$,
where ${u: A_1 \to A_2}$ is any homomorphism of commutative rings, which induces the homomorphism \linebreak ${\AutL(A_1) \to \AutL(A_2)}$ denoted by the same letter $u$, and elements $g, h$ are from
$ \AutL(A_1)$ such that $g (A_1[[t]]) \supset h (A_1[[t]])$.)

We will construct now explicitly a section of  the determinantal  central extension (the section not as group functors) and, hence, construct the corresponding $2$-cocycle on $\AutL$
with coefficients in $\GG_m$.

For this goal we construct another exact sequence of group functors that for any commutative ring $A$ looks as follows:
$$
1 \lrto \GL_f(A)  \lrto \widehat{\AutL}(A)  \lrto \AutL(A)  \lrto 1  \, \mbox{.}
$$

Here the group $\GL_f(A)$ consists of all elements $r \in \GL_A(A[[t]])$ such that there is an integer $n >0$, which depends on $r$,
with the property $r |_{t^n A[[t]]} = {\rm id}$. Here $\GL_A(A[[t]])$ is the group of all $A$-module automorphisms $A[[t]] \to A[[t]]$ (we do not demand that these automorphisms preserve the ring structure).
We note that $r$ is always a continuous map, since $r - {\rm id} $ is a continuous map.

The group  $\widehat{\AutL}(A)$ as a set consists of all pairs $(g, r)$, where $g \in \AutL(A)$ and $r \in \GL_A(A[[t]])$
such that there is an integer $n > 0$, which depends on $g$ and $r$, with the property $d_g |_{t^nA[[t]]} = r |_{t^nA[[t]]}$. We note that
$r$ is always a continuous map, since $d_g -r$ and $d_g$ are continuous maps. Using that $g$, $g^{-1}$ and $r$ are continuous maps, it is easy to see that the set
$\widehat{\AutL}(A)$ is a subgroup in $\AutL(A) \times \GL_A(A[[t]])$. This gives the group structure on $\widehat{\AutL}(A)$.

Now the map $\widehat{\AutL}(A)  \to \AutL(A)$ is $(g,r)  \mapsto g$.

We will construct a morphism of exact sequences of group functors that for any commutative ring $A$ looks as follows:
 $$
 \xymatrix{
 1 \ar[r] & \GL_f(A) \ar[d]  \ar[r] &  \widehat{\AutL}(A) \ar[d]^(.4){\pi}  \ar[r] & \AutL(A) \ar@{=}[d] \ar[r] & 1 \\
  1 \ar[r] & A^*   \ar[r] &  \widetilde{\AutL}(A)   \ar[r] & \AutL(A)  \ar[r] & 1
 }
 $$
where
all homomorphisms given by vertical arrows are surjective and $\pi |_{\GL_f(A)}$ equals to the map $\det $, i.e. the determinant, which is well-defined by  definition of $\GL_f(A)$ (cf. also~\cite[Prop.~1.6]{FZh}, where the similar morphism of complexes  was written when $A$ is a field).

Let $(g, r)  \in \widehat{\AutL}(A)$. Then there is an $A$-submodule $L =t^m A[[t]]  $ with $m \ge 0$
such that  $g(L)  \subset A[[t]]$ and  $r |_L = d_g |_L$.
Since $g(L) \subset A[[t]]$, we have $d_g |_L = g |_L$.  Therefore $r(L)= g(L)$. Then we have the following isomorphism
$$
\frac{A[[t]]}{g(L)}= \frac{A[[t]]}{r(L)}  \stackrel{r}{\longleftarrow}  \frac{A[[t]]}{L}
\stackrel{g^{-1}}{\longleftarrow} \frac{g(A[[t]])}{g(L)}
$$
Since all the above $A$-modules are projective, taking some $N =t^k A[[t]] \subset g(L)$, we obtain the isomorphism of $A$-modules
$g(A[[t]])/ N \to A[[t]]/N$, which gives the isomorphism of the corresponding top exterior powers of these projective $A$-modules. This gives an element  $ s \in \det(g (A[[t]])  | A[[t]])$ such that  for any prime ideal $P$ of $A$ we have  ${s \notin P\det(g (A[[t]])  | A[[t]])}$. Thus we have constructed the map $\pi$ as $\pi((g,r))= (g,s)$. This map  is a group homomorphism.

The morphism of functors $\widehat{\AutL}  \to \AutL$ has a canonical section (which is not, in general, a morphism of group functors). For a commutative ring $A$
this section is
$$
\AutL(A)  \lrto \widehat{\AutL}(A)  \quad : \quad g \longmapsto (g, d_g)  \, \mbox{.}
$$
The composition of this map with $\pi$ gives a section of the map \linebreak ${\widetilde{\AutL}(A) \to \AutL(A) }$. This section is not, in general,
a group homomorphism. In the usual way,  this gives a $2$-cocycle for the central extension~\eqref{det-cent-ext}
\begin{equation}  \label{coc-D}
D(f,g) = \det (d_f \cdot d_g  \cdot d_{fg}^{-1}  )  \, \mbox{,}
\end{equation}
where  $f,g \in \AutL(A) $,    $d_{fg}= c_f \cdot b_g + d_f  \cdot d_g$ (recall that we use  notation from formula~\eqref{matr}). We note that there is an integer $n \ge 0$ such that
$(d_f \cdot d_g \cdot d_{fg}^{-1})  |_{t^n A[[t]]}  = {\rm id}$. Therefore the determinant $\det$ in  formula~\eqref{coc-D}  is well-defined.

Thus, we have proved the following Proposition.
\begin{prop}
The determinantal central extension of group functor $\AutL$ by group functor $\GG_m$ is defined by
 a $2$-cocycle $D$ given by formula~\eqref{coc-D}.
\end{prop}

\begin{nt} \em
Formula~\eqref{coc-D} is an algebraic formal analog of  formula from Proposition~(6.6.4) of~\cite{PS} form the theory of smooth loop groups.
\end{nt}

\section{Corresponding cocycles on Lie algebras} \label{Sec-coc-Lie}

In Appendix A we collected statements on how to construct Lie algebra valued functors and Lie algebra $2$-cocycles from
group ind-schemes (or more generally, group functors with some conditions) and $2$-cocycles on group ind-schemes. Using these statements,
in this section we calculate the  $2$-cocycles   $\Lie \hat{B}$  and $\Lie D$
on Lie algebra valued functors
that correspond to the $2$-cocycles $\hat{B} $  and $D$ on the group { ind-scheme}
 $\AutL$
 with coefficients in $\GG_m$
from Sections~\ref{FB} and~\ref{DCE} correspondingly.

\subsection{Lie algebra of continuous derivations}  \label{Sec-der}
We calculate in this section the Lie algebra valued functor $\Lie \AutL$ of the group  ind-scheme $\AutL$.

Let $A$ be a commutative ring.
 It is easy to see (see Appendix~A.1--A.2)  that $\Lie \AutL (A) $  consists of elements $\mu \in \AutL (A[\varepsilon]/ (\varepsilon^2))$ such that the corresponding element $\widetilde{\mu} = \mu(t)$ from   $(A[\varepsilon]/(\varepsilon^2))((t))$ equals  $t + g \ve  $, where $g$ is any element from $A((t))$, see notation in Section~\ref{Autgr}.

 Hence we can identify the underlying  $A$-module of $\Lie \AutL (A) $ with  the $A$-module of all continuous $A$-derivations of the $A$-algebra
 $A((t))$. Indeed, for any ${f, g \in A((t))}$ we have in the ring $(A[\ve]/ (\ve^2))((t))$ by Taylor formula that
 $$
 f \circ (t +  g \ve)  = f + g f' \ve   \,  \mbox{.}
 $$
 Therefore the continuous $A[\ve]/ (\ve^2)$-automorphisms  $t \mapsto t + g \ve$ of the $A[\ve]/(\ve^2)$-algebra $(A[\ve]/ (\ve^2))((t))$  corresponds to continuous
 $A$-derivations $g \frac{\partial}{\partial t}$ of the $A$-algebra $A((t))$.

Now we calculate the Lie bracket in the Lie algebra $\Lie \AutL (A)$. For this goal we calculate in the ring
$(A[\ve_1, \ve_2]/ (\ve_1^2, \ve_2^2))((t))$ for $g_1, g_2  \in A((t))$:
$$
(t + g_1\ve_1 ) \circ (t + g_2 \ve_2 ) = t + g_2 \ve_2  + g_1 \ve_1  + g_2 g_1' \ve_1 \ve_2  \, \mbox{.}
$$
Hence we have that
$$
(t + g_2\ve_2 ) \circ (t + g_1 \ve_1 )   = (t + g_1\ve_1 ) \circ (t + g_2 \ve_2 ) \circ  (t + (g_1 g_2' - g_2 g_1') \ve_1 \ve_2)  \, \mbox{.}
$$

Therefore for any $\mu_1, \mu_2  \in \Lie \AutL (A) $  such that $\widetilde{\mu_i} = t + g_i \ve $ (where $i=1$ and $i=2$)  we obtain
(see formula~\eqref{ff} in Appendix A.2)    the Lie bracket ${[\mu_1, \mu_2]  \in \Lie \AutL (A)  }$  such that
$$
\widetilde{[\mu_1, \mu_2]} = t +  (g_1 g_2' -g_2 g_1') \ve_1 \ve_2   \, \mbox{.}
$$
The derivation $(g_1 g_2' -g_2 g_1') \frac{\partial}{\partial t}$ is equal to the derivation $\left[g_1 \frac{\partial}{\partial t}, \, g_2 \frac{\partial}{\partial t} \right]$, where $[\cdot, \cdot]$
is the usual commutator bracket on derivations.

Thus,  the Lie algebra $\Lie \AutL (A)$ over the ring $A$ is naturally identified with the Lie algebra of continuous $A$-derivations of the ring
 $A((t))$ with the usual Lie bracket for derivations. (We note that we also used
  formula~\eqref{tld}.)

\begin{nt} \em  \label{dence}
In the Lie $A$-algebra $\Lie \AutL (A)$ of continuous $A$-derivations of $A((t))$ we have the dense $A$-subalgebra with the basis $L_n = t^{n+1} \frac{\partial}{\partial t}$, $n \in \dz$, with the bracket ${[L_n, L_m] = (m -n) L_{n+m}}$.
\end{nt}

\subsection{$2$-cocycles  $\Lie \hat{B}$  and $\Lie D$  on Lie algebra valued functors}  \label{sec-theor1}

We note that $\Lie \GG_m  = \GG_a$, where $\GG_a(A) = A$ with the zero Lie bracket.

We calculate now explicitly the   $2$-cocycle $\Lie \hat{B}$ on Lie algebra valued functor $\Lie \AutL$
with coefficients in $\GG_a$ that corresponds to the $2$-cocycle $\hat{B}$ on the  group { ind-scheme}.

For any commutative ring $A$ we identify  the Lie algebra $\Lie \AutL (A)$ over the ring $A$ with the Lie algebra of continuous $A$-derivations of the ring
 $A((t))$.

\begin{prop}  \label{LieB} Let $A$ be a commutative ring.
For any  elements
$g_i \in A((t))$, where $i=1$ and $i=2$, we have
\begin{equation}  \label{formLieB}
\Lie \hat{B}  \left(g_1 \frac{\partial}{\partial t}, \, g_2 \frac{\partial}{\partial t} \right) = 2 \res (g_1' \cdot d g_2')  \, \mbox{,}
\end{equation}
where we recall that $g_i' = \frac{\partial g_i}{\partial t} $.
\end{prop}
\begin{proof}
Let $\Lie \hat{B} \left(g_1\frac{\partial}{\partial t}, \, g_2 \frac{\partial}{\partial t} \right)  = b \in A$.
Then,  according to Proposition~\ref{ex-form-coc} from Appendix A.3, the element $b $ is defined in the following way.

For every $i=1$ and $i=2$ we consider  the rings $E_i = A[\ve_i]/ (\ve_i^2)$ and elements
$\alpha_i \in \AutL(E_i((t)))$ such that $\widetilde{\alpha_i} = \alpha_i(t)= t + g_i \ve_i$.  We consider  $E = A[\ve_1, \ve_2]/ (\ve_1^2, \ve_2^2)$
and $E((t))$ that contains the rings $E_i((t))$.
Then we have
$$
\hat{B}(\alpha_1, \alpha_2) \cdot \hat{B}(\alpha_2, \alpha_1)^{-1}  =  1 + b \hspace{0.2pt} \ve_1 \ve_2  \in E^*    \, \mbox{.}
$$

By formula~\eqref{FBT}   we have
$$
\hat{B}(\alpha_1, \alpha_2 )=
\CC (\widetilde{\alpha_1}', \,  \widetilde{\alpha_2}'\circ \widetilde{\alpha_1})
= \CC(1 + g_1' \ve_1,   (1 + g_2' \ve_2) \circ (t + g_1 \ve_1))  \, \mbox{.}
$$

Using the Taylor formula, it is easy to see that
$$
(1+ g_2' \ve_2) \circ (t + g_1 \ve_1)= (1 + g_2' \ve_2)(1 + g_1 g_2'' \ve_1 \ve_2)
$$

We note that from the functoriality of the Contou-Carr\`ere symbol $\CC$ we have in the ring
$E((t))$ for any $d_1, d_2 \in A((t))$  an equality $\CC(1+ d_1 \ve_1, 1 + d_2 \ve_2) = 1 + d \ve_1 \ve_2$ for some $d \in A$.
Hence, using the $A$-algebra endomorphism $E \to E$, $\ve_1 \mapsto \ve_1$, $\ve_2  \mapsto \ve_1 \ve_2$, which induces the map of
$\CC(1 + g_1' \ve_1, 1 + g_1 g_2'' \ve_2)$ to
$
\CC(1 + g_1' \ve_1, 1 + g_1 g_2'' \ve_1 \ve_2)
$, we obtain that $
\CC(1 + g_1' \ve_1, 1 + g_1 g_2'' \ve_1 \ve_2)  =1
$.
Therefore and using the bimultiplicativity of $\CC$, we have that
$$
\hat{B}(\alpha_1, \alpha_2 )  =\CC(1 + g_1' \ve_1, \, 1 + g_2' \ve_2)  = 1 + \res (g_1' \cdot  d g_2') \ve_1 \ve_2  \, \mbox{,}
$$
where the last equality follows at once from formula~\eqref{CC-exp-log} when $\Q \subset A$ and from formula~\eqref{CC-form}  in the general case.

Hence we also have $\hat{B}(\alpha_2, \alpha_1 )^{-1} = 1 - \res (g_2' \cdot  d g_1') \ve_1 \ve_2 = 1+ \res (g_1' \cdot  d g_2') \ve_1 \ve_2  $.
Therefore $b = 2 \res (g_1' \cdot  d g_2')$.
\end{proof}

\begin{nt} \em
Proposition~\ref{LieB} is an algebraic formal analog of the   statement from the theory of infinite-dimensional Lie groups  that the Lie algebra $2$-cocycle that corresponds to the Bott-Thurston group $2$-cocycle is the doubled Gelfand-Fuks $2$-cocycle on the Lie algebra of smooth vector fields on the circle, see, e.g., definition-proposition~2.4 from chapter II of~\cite{KW}.
\end{nt}

We consider an element $s$ from the Lie algebra $\Lie \AutL (A)$, which is the Lie algebra of continuous $A$-derivations of $A((t))$, as a block matrix (cf. formula~\eqref{matr})
\begin{equation}  \label{matr2}
\begin{pmatrix}
a_s & b_s  \\
c_s & d_s
\end{pmatrix}
\end{equation}
with respect to the decomposition $A((t))= t^{-1} A[t^{-1}] \oplus A[[t]]$, and where the matrix acts  on an element-column from $A((t))$ on the left.

The following proposition is an algebraic formal analog of the statement from the theory of smooth loop groups,
see~\cite[Prop.~6.6.5]{PS}.
\begin{prop}
Let $A$ be a commutative ring.
For any $s,  r$ from $\Lie \AutL (A)$ we have
\begin{equation}  \label{formLieD}
\Lie D (s, r) =  \tr (  c_r \cdot b_s  - c_s \cdot b_r  )  \, \mbox{,}
\end{equation}
where the trace of the map  from $A[[t]]$ to $A[[t]]$
is well-defined, since there is an integer $n \ge 0$ such that $b_r$ and $b_s$  restricted to $t^n A[[t]]$ equal zero.
\end{prop}
\begin{proof}
From formula~\eqref{coc-D} by direct calculations with $D$ over the ring $A[\ve_1, \ve_2]/ (\ve_1^2, \ve_2^2)$
it is easy to see   that
$$
D({\rm id} + s \ve_1, {\rm id}+ r \ve_2) = \det({\rm id} - c_s b_r \ve_1 \ve_2 )= 1 - \tr(c_s b_r) \ve_1 \ve_2  \, \mbox{,}
$$
where $\rm id$ is the identity map.
Hence we obtain that
$$
D({\rm id} + s \ve_1, {\rm id}+ r \ve_2) \cdot D({\rm id} + r \ve_2, {\rm id}+ s \ve_1)^{-1}  =  1 + \tr ( c_r \cdot b_s  - c_s \cdot b_r  ) \ve_1 \ve_2 \, \mbox{,}
$$
Now, by Proposition~\ref{ex-form-coc} from Appendix A.3, we obtain the statement.
\end{proof}

Now we can compare $2$-cocycles $\Lie D$ and  $\Lie \hat{B}$   on the Lie algebra valued functor $\Lie \AutL$ with coefficients in $\GG_a$.

\begin{Th}  \label{theor-1}
We have the following equality of $2$-cocycles on the Lie algebra valued functor $\Lie \AutL$ with coefficients in $\GG_a$:
$$
12 \Lie D  = \Lie  \hat{B}   \, \mbox{.}
$$
\end{Th}
\begin{proof}
Let $A$ be a commutative ring.
From formulas~\eqref{formLieB} and~\eqref{formLieD} it is easy to see that
 $\Lie \hat{B}$ and $\Lie D$ are    continuous maps in each argument from ${A((t)) \times A((t))}$
to $A$. (Here the continuous map in each argument means that we fix one argument of the map then the resulting  map   is continuous in another argument. Besides, we note that in formula~\eqref{formLieD} both maps $\tr ( c_s \cdot b_r)$ and $\tr (c_r \cdot b_s) $ are continuous in each argument.)
 Therefore  the maps
 $\Lie \hat{B}$ and $\Lie D$
 are uniquely defined by values on pairs of elements: $L_m = t^{m+1} \frac{\partial}{\partial t}$  and  $L_n = t^{n+1} \frac{\partial}{\partial t}$.

By direct calculation with formula~\eqref{formLieB} we obtain
$$
\Lie \hat{B} \left( L_m, L_n \right) = -2 (m - m^3) \cdot \delta_{n, -m}  \, \mbox{,}
$$
where $\delta_{i,j} =1$ if $i=j$, and $\delta_{i,j} =0$ if $i \ne j$.

We write $L_n$ as an infinite matrix $({L_{n,}}_{ij})$, where $L_n (t^j) = \sum_i {L_{n,}}_{ij} t^i$. The only matrix elements that are non-equal to zero are ${L_{n,}}_{ij}= j$ when $i = j+n$. Therefore we obtain that if $m \le 0$, then
$$\tr ( c_{L_n} \cdot b_{L_m}) = \delta_{n,-m} \cdot \sum_{j=0}^n  j (j -n) =   \frac{m^3 -m}{6}  \cdot \delta_{n, -m}$$
and $ \tr(  c_{L_m} \cdot b_{L_n}) =0$  (see notation in~\eqref{matr2}). If $m \ge 0$, then
$$\tr ( c_{L_m} \cdot b_{L_n}) = \delta_{n,-m} \cdot \sum_{j=-n}^0  j (j -n) =  \frac{m - m^3}{6}   \cdot  \delta_{n, -m}$$
and $ \tr(  c_{L_n} \cdot b_{L_m}) =0$. Hence and from formula~\eqref{formLieD} we obtain
$$
12 \Lie D \left( L_m, L_n \right) = - 12 \cdot \frac{m - m^3}{6}  \cdot \delta_{n, -m} = \Lie \hat{B} \left( L_m, L_n \right)  \, \mbox{.}
$$
\end{proof}

\section{Comparison of central extensions}  \label{Sec-comp}

We will assume further in this section that all the commutative rings are $\Q$-algebras.

Correspondingly, we will assume in this section that all the ind-schemes are defined over $\Q$ and functors are functors from the category of $\Q$-algebras.
For an ind-scheme $G$ defined over $\dz$ (or for the functor, which we denote by the same letter and which is representable by this ind-scheme)    we denote by $G_{\Q}$ the ind-scheme (or the corresponding functor) obtained by the extensions of scalars to $\Q$ (the functor restricted  to the category of $\Q$-algebras).

By a group ind-scheme $G$ we can construct the corresponding Lie algebra valued functor $\Lie G$ (see Appendix A) and the corresponding Lie algebra $ \Lie G(\Q)$ over $\Q$.

\subsection{Infinitesimal formal groups}  \label{Sec-form}

By {\em an infinitesimal  formal group  over $\Q$} we mean a group ind-scheme \linebreak ${G = \mbox{``$\varinjlim\limits_{i \in I}$''} \Spec A_i}$
 such that every $A_i$ is a finite-dimensional $\Q$-algebra and
 the corresponding profinite algebra of regular functions $\oo(G)= \varprojlim_{i \in I} A_i$ is a local $\Q$-algebra with the residue field $\Q$, see more
 on this notion, e.g,  in~\cite{Die}.

Now the functor
$G \mapsto \Lie G (\Q)$ gives an equivalence of the category of infinitesimal formal groups over  $\Q$ and the category of Lie
algebras over $\Q$ (which can be infinite-dimensional over $\Q$). Moreover, this is also true over any ground field of  characteristic zero,
see  section 3.3.2 in Expose $\rm VII_B$ written by P. Gabriel in~\cite{SGA3}.

In our case we have  the natural example of infinitesimal formal group. We consider the group functor ${\ff\AutL}_{\Q}$ such that
for any $\Q$-algebra $A$ the subgroup
$
{\ff\AutL}_{\Q}(A)
$
consists of elements $\varphi$ of the group $  \AutL_{\Q}(A)$ such that ${\widetilde{\varphi} = \varphi(t) \in A((t))}$
is an element of the following kind:
$$
t + \sum_i c_i t^i  \, \mbox{,}
$$
where all elements $c_i \in A$ are nilpotent elements and they are equal zero except for a finite number of elements.
(The statement  that ${\ff\AutL}_{\Q}$ is a  group subfunctor of $  \AutL_{\Q}$ it is easy to see. Indeed, the most nontrivial is to prove that if
 $\phi \in {\ff\AutL}_{\Q}(A)$, then $\phi^{-1}  \in {\ff\AutL}_{\Q}(A)$. This follows from the proof of Proposition~\ref{Decomp}, since
 $\phi = \phi_- \phi_+ $, where $\phi_- \in \AutLm_{\Q}(A)  \subset {\ff\AutL}_{\Q}(A)  $ and  the element  $\phi_+ $ belongs to the subset
 $ \AutLp_{\Q}(A) \cap {\ff\AutL}_{\Q}(A)$ of the group  $\AutL_{\Q}(A)$ and this subset is a subgroup.)

It is easy to see that the group functor ${\ff\AutL}_{\Q}$ is represented by the following ind-scheme
\begin{equation}  \label{ind-pr}
\mbox{``$\varinjlim\limits_{\{\epsilon_i\}}$''} \Spec \Q [c_i; i \in \dz]/ I_{\{\epsilon_i\}}  \, \mbox{,}
\end{equation}
where $\Q [c_i; i \in \dz]$ is the polynomial ring over $\Q$ on a set of variables $c_i$ with $i \in \dz$, and the limit is taken over all the sequences $\{\epsilon_i\}$ with $i \in \dz$ and $\epsilon_i$ are nonnegative integers such that all but finitely many $\epsilon_i$ equal zero, the ideal $I_{\{\epsilon_i\}}$ is generated by elements $c_i^{\epsilon_i +1}$ for all $i \in \dz$. Thus, ${\ff\AutL}_{\Q}$ is an infinitesimal formal group over $\Q$.

Clearly, we have the natural homomorphism  between the group ind-schemes \linebreak ${\ff\AutL}_{\Q}  \to \AutL_{\Q}$ that is an embedding on $A$-points for any commutative $\Q$-algebra $A$.

Now it is easy to see that the Lie algebra $\Lie \ff\AutL_{\Q} (\Q)$ over $\Q$ has a basis $L_n$ with $n \in \dz$
and the Lie bracket is $[L_n, L_m] = (m-n) L_{m+n}$, where $L_n = t^{n+1} \frac{\partial}{\partial t}$ is a continuous derivation of the algebra $\Q((t))$, cf. Remark~\ref{dence}.

\subsection{Central extensions of $\AutL_{\Q}$ by ${\GG_m}_{\Q}$}  \label{Sec-centr}

For a commutative group ind-scheme $F$ and a    group ind-scheme $G$, an element from $H^2(G, F)$ is an equivalence class of  central extensions of $G$ by $F$ that allow a section, see Remark~\ref{central}.
A central extension from this class defines a group ind-scheme and it gives a central extension of corresponding Lie algebras over $\Q$. Therefore we obtain the homomorphism
from  $H^2(G, F)$ to $H^2(\Lie G (\Q)), \Lie F(\Q))$ given on $2$-cocycles by  formula from Proposition~\ref{ex-form-coc}  from Appendix~A.3.

The following theorem is an algebraic analog of the corresponding statement in the theory of infinite-dimensional Lie  groups, see Corollary~(7.5) from~\cite{Se}.

\begin{Th}  \label{Th-2-coc}
An element from $H^2(\AutL_{\Q}, {\GG_m}_{\Q})$, where $\AutL_{\Q}$ acts trivially on ${\GG_m}_{\Q}$, is uniquely defined by its image in
$H^2(\Lie \AutL (\Q), \Q )$ together with  its  restriction  to $H^2 ( \AutLp_{\Q} , {\GG_m}_{\Q})$ (see Section~\ref{more}).
\end{Th}

\begin{proof}
Let we have a central extension of group functors
\begin{equation}  \label{CE}
1 \lrto {\GG_m}_{\Q}  \lrto  H  \stackrel{\pi}{\lrto}  \AutL_{\Q}  \lrto  1
\end{equation}
such that there is a section $p$ of morphism $\pi$ (and, in general, this section is not a morphism of group functors). Then it is enough to prove that
if the corresponding central extension of Lie algebras over $\Q$ is isomorphic to the trivial central extension and the restriction of  central extension~\eqref{CE}
to the group ind-scheme  $\AutLp_{\Q}$ is isomorphic to the trivial central extension, then   central extension~\eqref{CE} is isomorphic to the trivial central extension. We will prove this statement in several steps.

{\em Step $1$.}
Central extension~\ref{CE} and a section $p$
 give the $2$-cocycle
 $$K \; : \;  \AutL_{\Q}  \times \AutL_{\Q}   \to {\GG_m}_{\Q}   \, \mbox{,}$$
 which is a morphism of ind-schemes. By changing the section
 $p$ to the new section ${p'=p \cdot p(e)^{-1}}$, where $e$ is the identity element  of the group ind-scheme  $ \AutL_{\Q}$, we see that $p'(e)$
 is the identity element of the group ind-scheme $H$
 and therefore the new $2$-cocycle $K'$ constructed by $p'$ satisfies $K'(e,e)=1$, where $1$ is the identity element of the group scheme $\GG_m$.

 We consider the restriction of central extension~\eqref{CE} to the group ind-scheme $\ff\AutL_{\Q}$:
 \begin{equation}  \label{Centr2}
 1 \lrto {\GG_m}_{\Q} \lrto \pi^{-1}(\ff\AutL_{\Q})  \stackrel{\tau}{\lrto} \ff\AutL_{\Q}  \lrto 1 \, \mbox{.}
 \end{equation}
 From the condition $K'(e,e)=1$ and by the definition  of  $\ff\AutL_{\Q}$,   we obtain that for a commutative $\Q$-algebra $A'$ such that
 ${\rm Nil}(A') =0$ the $2$-cocycle $K'$ restricted to
 $${\ff\AutL_{\Q}(A')  \times \ff\AutL_{\Q}(A') \, =  \, e \times e}$$
  equals $1 \in {A'}^*$.
 Therefore, by considering the homomorphism $A \to A/ {\rm Nil} A$, we obtain that the $2$-cocycle $K'$ restricted to ${\ff\AutL_{\Q}  \times \ff\AutL_{\Q}}$ takes values in the group ind-scheme
 ${\widehat{{\GG_m}_{\Q}}}$, where  ${\widehat{{\GG_m}_{\Q}}(A) = 1 + {\rm Nil}(A) \subset A^* }$
 for any commutative \linebreak $\Q$-algebra $A$.
 Thus, ${\widehat{{\GG_m}_{\Q}}
  =  1 + \widehat{{\GG_a}_{\Q}} }$, where the group ind-scheme \linebreak
 ${\widehat{{\GG_a}_{\Q}} = \mbox{``$\varinjlim\limits_{n \ge 0}$''} \Spec \Q[t]/ t^n}$.

 Therefore central extension~\eqref{Centr2} comes from another central extension
 \begin{equation}  \label{Centr3}
 1 \lrto \widehat{{\GG_m}_{\Q}} \lrto T \stackrel{\kappa}{\lrto} \ff\AutL_{\Q}  \lrto 1 \, \mbox{,}
 \end{equation}
   by means of the natural morphism $\widehat{{\GG_m}_{\Q}} \to {\GG_m}_{\Q}$. Since there is (in general, non-group) section of $\kappa$, we have that   $T \simeq \widehat{{\GG_m}_{\Q}} \times \ff\AutL_{\Q} $ as an ind-scheme. Since $\widehat{{\GG_m}_{\Q}}$,  $\ff\AutL_{\Q}$ and $T$ are   infinitesimal  formal groups  over $\Q$, the isomorphism class of
  central extension~\eqref{Centr3} is uniquely defined by the  isomorphism class of corresponding central Lie algebra extension, i.e. by
  the image of
  the element  from  $H^2(\Lie \AutL (\Q), \Q) $. By our condition, the last element equals zero. Therefore, $ T \simeq \widehat{{\GG_m}_{\Q}} \times \ff\AutL_{\Q}$ as a group ind-scheme (where the last decomposition of group ind-schemes may differ from the decomposition of ind-schemes written just above).
  Hence there is a morphism of group ind-schemes
  $$s \, :  \, \ff\AutL_{\Q}  \lrto \pi^{-1}(\ff\AutL_{\Q})$$ that is a group section of $\tau$.

  {\em Step $2$.}
  By our condition, central extension~\eqref{CE} restricted to $\AutLp_{\Q}$ is split. This means that there is a group section of morphism $\pi$ over $\AutLp_{\Q}$:
  $$
  r \; : \; \AutLp_{\Q}  \lrto \pi^{-1}(\AutLp_{\Q})  \, \mbox{.}
  $$

  By sections $r$ and $s$ we construct now a section $q$ of morphism $\pi$.
  By Proposition~\ref{Decomp}, for any commutative $\Q$-algebra $A$ and any element $g \in \AutL(A)$ there is a unique decomposition
  $g = g_{+} \cdot g_{-} $, where $g_+ \in \AutLp(A)$ and $g_{-} \in \AutLm(A)$. We define
  $$q(g)= r (g_+) \cdot s(g_-)  \in H(A) \, \mbox{,}$$
  where we use that $g_- \in \ff \AutL_{\Q}(A)$.

  We will prove that $q$ restricted to $\ff \AutL$ coincides with $s$. Suppose that $g$ is from  $\ff \AutL_{\Q}(A)$.
  Then we have that $g_-$ and $g_+ = g \cdot g_{-}^{-1}$
  are from    $\ff \AutL_{\Q}(A)$. Since $s$ is a group morphism, it is enough to prove that $r =s$
 on the group ind-scheme $V$, where
   $$ V{ (A ) = \AutLp_{\Q}(A) \,  \cap  \, \ff\AutL_{\Q}(A)}$$
  for
 any   commutative $\Q$-algebra $A$  and the intersection is taken inside $\AutL_{\Q}(A)$. We consider the group morphism
 $$
 r/s \; : \; V \lrto {\GG_m}_{\Q}  \, \mbox{.}
 $$
 As in Step $1$, by considering the homomorphism $A \to A/ {\rm Nil} A$, we obtain that $r/s$ takes values in $\widehat{{\GG_m}_{\Q}}$, i.e. we have
 $$
 r/s \; : \; V \lrto \widehat{{\GG_m}_{\Q}}  \, \mbox{.}
 $$

It is easy to see that $V$ is an infinitesimal formal group (to represent $V$ it is enough to take in formula~\eqref{ind-pr} only indices with $i \ge 0$). Therefore the group morphism $r/s$ is uniquely defined by the corresponding homomorphism of Lie algebras over $\Q$:
$$
\Lie (r/s) \; : \; \Lie V (\Q) \lrto \Q = \GG_a(\Q)=\Lie \widehat{{\GG_m}_{\Q}}(\Q) \, \mbox{.}
$$
It is easy to see that the Lie algebra $\Lie V (\Q)$ has a basis $L_n$, where $n \ge -1$, over $\Q$, and the Lie bracket is
$[L_n, L_m]= (m-n)L_{n+m}$. Therefore we have that
$$[\Lie V (\Q), \, \Lie V (\Q)] = \Lie V (\Q) \mbox{.} $$
Since $\GG_a$ is an Abelian Lie algebra, we obtain that $\Lie (r/s) =0$. Therefore $r/s = 1$. Hence $r =s$ on $V$, and we obtain
 that $q$ restricted to $\ff \AutL$ coincides with~$s$.

{\em Step 3.} We prove now that the section $q$ is a group section, i.e. it is a morphism of group ind-schemes.

We consider a morphism of ind-schemes:
$$
\beta \; : \; \AutL_{\Q}  \times \AutL_{\Q}  \lrto {\GG_m}_{\Q}   \, \mbox{,}
$$
where for any commutative $\Q$-algebra $A$  and any elements  $g_1$ and $g_2$  from $\AutL(A)$,
 by definition,
 $$
 \beta(g_1 \times g_2) = q(g_1) \cdot q(g_2)   \cdot q(g_1g_2)^{-1}  \in A^*  \, \mbox{.}
 $$

We have the natural morphism
$$ U =\ff \AutL_{\Q}  \times \ff \AutL_{\Q}  \stackrel{\gamma}{\lrto}  \AutL_{\Q}  \times \AutL_{\Q} \, \mbox{.} $$
We note that the composition of morphisms   $\beta \cdot \gamma$  is a constant morphism from $U$ to ${\GG_m}_{\Q}$ that is  equal to $1$, because $q$ restricted to
$\ff \AutL$ coincides with~$s$, and $s$ is a group morphism.

The composition of morphisms $\beta \cdot \gamma$  corresponds to the composition of  homomorphisms of $\Q$-algebras  of regular functions
$$
\Q[x, x^{-1}]  \stackrel{\beta^*}{\lrto}  \oo \left(\AutL_{\Q}  \times \AutL_{\Q} \right)  \stackrel{\gamma^*}{\lrto}  \oo(U) \, \mbox{.}
$$
Now we have $\gamma^* \beta^* = \gamma^* {{\mathbf 1}^*}$, where $\mathbf{1}$ is the constant morphism that is equal to $1$. Therefore
$\beta^* = {\mathbf 1}^*$ and
$\beta = {\mathbf 1}$,
because $\gamma^*$ is an embedding. The last statement follows from the explicit forms of $\Q$-algebras of regular functions $\oo(\AutL_{\Q})$ and
$\oo(\ff \AutL_{\Q})$, see formula~\eqref{ind-sch}, where we have to change the ground ring $\dz$ to the ground field $\Q$, and formula~\eqref{ind-pr}. Now $\gamma^*$ is the embedding of the  $\Q$-algebra of mixture of  polynomials, Laurent polynomials and formal power series in infinite number of variables to the $\Q$-algebra of formal power series in infinite number of variables, where in notation of formulas~\eqref{ind-sch} and~\eqref{ind-pr} we have that $a_i \mapsto c_i$, where $i \ne 1$, and $a_1 \mapsto 1 + c_1$.

Thus, since $\beta = {\mathbf 1}$, we obtain that $q$ is a group section. Therefore central extension~\eqref{CE}
is isomorphic to the trivial central extension.

\end{proof}

\subsection{Comparison of two central extensions given by $\hat{B}$ and $D$}   \label{Sec-main}

As corollary of Theorem~\ref{Th-2-coc} we obtain the following theorem
\begin{Th}  \label{theor-3}
In the group $H^2(\AutL_{\Q}, {\GG_m}_{\Q})$ we have
$$
D^{12}= \hat{B}    \, \mbox{,}
$$
where $\hat{B}$ is the formal Bott-Thurston cocycle and $2$-cocycle $D$ defines the determinantal central extension.
\end{Th}
\begin{proof}
By Theorem~\eqref{theor-1}  we have that $12 \Lie D = \Lie \hat{B} $. Hence we obtain that this equality  is true after the application to the Lie algebra $\Lie \ff \AutL_{\Q}(\Q) $.
By formulas~\eqref{FBT}, \eqref{CC-form}  (or, simpler, for $\Q$-algebras, by~\eqref{CC-exp-log} ) and by~\eqref{coc-D}, $2$-cocycles $\hat{B}$ and $D$ restricted to $\AutLp \times \AutLp$ are the constant morphisms that are equal to $1$.
Now we apply Theorem~\ref{Th-2-coc}.
\end{proof}

\begin{nt} \label{rem-Q} \em
There is  an interesting question. Is the statement of Theorem~\ref{theor-3} still true in the group $H^2(\AutL, {\GG_m})$,
i.e. when we consider any commutative rings $A$, not only $\Q$-algebras?
\end{nt}

\begin{nt} \label{Quest-2} \em
There is another interesting question. In the proof of Theorem~\ref{theor-3} (which is based on Theorem~\ref{Th-2-coc})
we obtained only that $2$-cocycles $\hat{B}$ and $D$ equal modulo the group of  $2$-coboundaries. Is it true that $\hat{B} = D$ as morphisms of group ind-schemes defined over~$\Q$?

We note that if this is true, then $\hat{B} =D$ for any commutative rings $A$, which are not necessarily $\mathbb Q$-algebras, because these $2$-cocycles are morphisms from the ind-scheme $\AutL \times \AutL$, which is the inductive limit of flat over $\dz$ schemes. This would also give the positive answer for the question in Remark~\ref{rem-Q}.
\end{nt}

\begin{nt} \em
In~\cite[pages~331-332]{Se} it is claimed an analog of Theorem~\ref{theor-3} for central extensions of the group ${\rm Diff} \, S^1$. This analog is not true in~\cite{Se}.
Namely, G.~Segal compares the Bott-Thurston $2$-cocycle (which is written a little bit in another way than in Section~\ref{BT}, since usually one supposes that the group ${\rm Diff} \, S^1$ acts on $S^1$ on the left) and the $2$-cocycle coming from the determinantal central extension of ${\rm Diff} \, S^1$, see also~\cite[\S~6.8]{PS}. For this goal he uses an infinite-dimensional Lie group  analog of Theorem~\ref{Th-2-coc}. G.~Segal explicitly calculates that the corresponding  $2$-cocycles on the Lie algebra of vector fields ${\rm Vect} \, S^1$ are proportional and find  the proportionality  coefficient. Then, by construction, the determinantal central extension restricted to the subgroup ${\rm PSL}(2, \dr)$, which is the group of holomorphic automorphisms of the unit disk in $\dc$, is trivial.
But the Bott-Thurston cocycle restricted to ${\rm PSL}(2, \dr)$ gives a non-trivial central extension, see Appendix
written by R.~Brooks in~\cite{B2} and page~229 in \S~4.5.1 of~\cite{GR}. This contradicts to what is written in~\cite{Se}!

\end{nt}

\section{Ringed spaces, gerbes and part of formal Riemann-Roch theorem}  \label{sec-ringed}

\subsection{$\oo((t))$-spaces}  \label{sec-spaces}

In this section  we give the implication of Theorem~\ref{theor-3} for the part of formal Riemann-Roch theorem  (see Theorem~\ref{theor-4}) for some ringed space
via equality of some cohomology classes in  the  \v{C}ech cohomology group $\check{H}^2(S, \oo_S^*)$   for a separated scheme $S$ over $\Q$.

For any scheme $V$ by $\oo_V((t))$ we denote the sheaf of $\oo_V$-algebras on $V$  associated with the presheaf
$U \mapsto \oo_V(U) ((t))$, where $U \subset V$ is an open subset. Since any affine scheme is quasi-compact, for any open affine subscheme
$W \subset V$ we have ${\oo_V((t)) (W) = \oo_V(W)((t))}$.

For any open cover $S = \bigcup_{i \in I} U_i $, denote $U_{i_0 \ldots i_k}= U_{i_0} \cap \ldots \cap U_{i_k}$, where $i_0, \ldots, i_k \in I$.

\begin{defin}  \label{def2}
Let $S$ be any scheme. By $\oo((t))$-space $\Theta$ over $S $ we call  a ringed space $(S, \M)$, where $\M$
is the sheaf of $\oo_S$-algebras on $S$, together with the equivalence class of the  following data.
There is an open cover $S = \bigcup_{i \in I} U_i $  such that for any $i \in I$
there is an isomorphism
\begin{equation}  \label{good-cov}
\phi_i \; : \; \M |_{U_i} \lrto \oo_{U_i} ((t))
\end{equation}
of sheaves of $\oo_{U_i}$-algebras
such that for any $i, j \in I$ and any open affine subscheme ${V \subset U_{ij}}$ the (transition)  automorphism
$ \phi_{ij}= \phi_i \phi_j^{-1}$ of the sheaf $\oo_{U_{ij}} ((t))$ on $U_{ij}$
induces the continuous automorphism of $\oo_S(V)$-algebra $ \oo_S(V) ((t))$.
Two such data are equivalent if their union is again such a data.
\end{defin}

\begin{nt} \em
Locally on $S$ an  $\oo((t))$-space corresponds to the punctured formal neighbourhood of a section of a  smooth morphism to $U$ of relative dimension $1$, where an open subset $U \subset S$
(see~\cite[Corollary 16.9.9, Theorem 17.12.1]{EGA4}).
\end{nt}

Now we have the Picard group
$${\rm Pic} ((S, \M))=  \check{H}^1(S, \M^*) = H^1(S, \M^*)$$
of isomorphism classes of
locally free
 sheaves of $\M$-modules of rank $1$ on the topological space $S$. (Here by $\check{H}^*(\cdot, \cdot)$ we denote the \v{C}ech cohomology.)
 For any locally free
 sheaf $\ff$ of $\M$-modules of rank $1$ on the topological space $S$ by $c_1(\ff)$
 we denote the class of this sheaf in ${\rm Pic} ((S, \M))$.

Let $\ff$ and $\g$ be two locally free
 sheaves of $\M$-modules of rank $1$ on the topological space $S$. We consider an open cover $S = \bigcup_{i \in I} U_i$
of $S$ such that $\ff |_{U_i} \simeq  \g |_{U_i} \simeq \M|_{U_i}$ for any $i \in I$.

For any $i, j \in I$ let $f_{ij}$  and $g_{ij}$ be elements from $\M^*(U_{ij})$
that give \v{C}ech $1$-cocycles for the sheaves $\ff$ and $\g$ correspondingly.

We have the \v{C}ech $2$-cocycle for the sheaf $\M^* \otimes_{\mathbb Z} \M^*$ on $S$ given for any $i,j,k \in I$ by the element
\begin{equation} \label{cup-Cech}
f_{ij} \otimes g_{jk}
\in \mathop{\M^* \otimes_{\mathbb Z} \M^*} (U_{i j k})  \, \mbox{,}
\end{equation}
where we consider the images of elements  $f_{ij}$ and  $g_{jk}$ in $\M^* (U_{ijk})$.

By  $K_2^M(\M)$ we denote the sheaf on $S$ associated with the presheaf ${U \mapsto K_2^M(\M(U)))}$.

Now we have the composition of the following maps
\begin{equation}  \label{K2-seq}
 H^1(S, \M^*) \otimes_{\mathbb Z} H^1(S, \M^*)  \lrto \check{H}^2(S,\M^* \otimes_{\mathbb Z} \M^*) \lrto \check{H}^2(S, K_2^M(\M))  \hookrightarrow  H^2(S, K_2^M(\M))  \, \mbox{,}
\end{equation}
where the first map is the $\cup$-product given on \v{C}ech cocycles by  formula~\eqref{cup-Cech}, and the second map is induced be the natural homomorphism of sheaves
$ \M^* \otimes_{\mathbb Z} \M^*  \lrto K_2^M(\M)$.

\begin{nt} \em
On a smooth algebraic surface $X$ over a field $k$, by the results of S.~Bloch and K.~Kato (see~\cite{Bl} and~\cite{Ka}) we have
$$
H^2(X, K_2(\oo_X)) = H^2(X, K_2^M(\oo_X)) = {\rm CH}^2(X)  \, \mbox{.}
$$
Then the formula like  formula~\eqref{K2-seq}  gives the intersection of divisors on $X$ (see~\cite[Remark~2.14]{Bl}):
$$
{\rm Pic} X   \otimes_{\mathbb Z} {\rm Pic} X \lrto {\rm CH}^2(X)  \, \mbox{.}
$$
\end{nt}

\subsection{Explicit \v{C}ech $2$-cocycles}  \label{sec-coc}

For any $\oo((t))$-space $\Theta = (S, \M)$ over $S$ we have the natural homomorphism ${\M^* \otimes_{\Z}  \M^*  \to \oo_S^*}$  of sheaves of Abelian groups  on $S$
induced by the Contou-Carr\`{e}re symbol $\rm CC$ in the following way. For any open $U_i \subset S$ and isomorphism~$\phi_i$ from~\eqref{good-cov},
we consider the composition of homomorphisms of sheaves of Abelian groups
$$ \M^* \otimes_{\Z} \M^* |_{U_i}  \lrto  \oo_{U_i}((t))^* \otimes_{\Z} \oo_{U_i}((t))^*  \lrto   K_2^M(\oo_{U_i} ((t))) \lrto \oo_{U_i}^*  \, \mbox{,} $$
 where the the first homomorphism  is induced by $\phi_i$, and the composition of the second and the third homomorphisms is induced by the Contou-Carr\`{e}re symbol $\rm CC$.
 These maps are glued together   correctly on $S$, since the Contou-Carr\`{e}re symbol is invariant
under the continuous automorphisms of the $\oo_S(U_{ij})$-algebra  $\oo_S(U_{ij}) ((t))$. The constructed homomorphism of sheaves induces the homomorphism $\partial$ given as composition of maps
$$
\partial \; : \;
\check{H}^2(S, \M^* \otimes_{\Z} \M^*) \lrto
\check{H}^2(S, K_2^M(\M))  \lrto \check{H}^2(S, \oo_S^*)  \, \mbox{, }
$$
which does not depend on the choice of the open cover $\{ U_i \}$ and the isomorphisms $\phi_i$ from the equivalence class of data.

\medskip

For any $\oo((t))$-space $\Theta = (S, \M)$ over $S$ we consider the sheaf $\Omega^1_{\M / \oo_S}$ of relative differential forms of $\Theta$ over $S$.
We will define  the quotient sheaf $\widetilde{\Omega}^1_{\Theta} = \Omega^1_{\M / \oo_S} / {\mathcal N}$ that will be a locally free
 sheaves of $\M$-modules of rank $1$ on the topological space $S$.
For any open $U_i \subset S$ and isomorphism~$\phi_i$ from~\eqref{good-cov} we define ${\mathcal N} |_{U_i}$ as the kernel of the composition of the following maps
$$
{\Omega}^1_{\M |_{U_i} / \oo_{U_i}}  \lrto {\Omega}^1_{\oo(U_i)((t)) / \oo_{U_i}}  \lrto
\widetilde{\Omega}^1_{\oo(U_i) ((t))}   \, \mbox{,}
$$
where the first map is induced by $\phi_i$ and the second map comes from formula~\eqref{form-ker}. The sheaves ${\mathcal N} |_{U_i}$ on $U_i$ are glued together to the sheaf ${\mathcal N}$ on $S$, which does not depend on the choice of the open cover $\{ U_i \}$ and the isomorphisms $\phi_i$
from the equivalence class of data.

For a $\oo((t))$-space $\Theta$ over $S$ we will be interested in the following element
\begin{equation}  \label{expl}
\partial \, (c_1(\widetilde{\Omega}^1_{\Theta}) \cup c_1(\widetilde{\Omega}^1_{\Theta}) )  \, \in \, \check{H}^2(S, \oo_S^*) \, \mbox{,}
\end{equation}
whose cohomology class depends only on $\Theta$.

\begin{nt}  \em  \label{dif-coc}
For a fixed open cover $\{ U_i\}$ and a data as  in~\eqref{good-cov},  we constructed the  explicit the \v{C}ech $2$-cocycle in formula~\eqref{expl}, since for any $U_i$ there is the natural section $d(\phi_i^{-1}(t))$
of  $\widetilde{\Omega}^1_{\Theta} |_{U_i}$, and this  gives the natural \v{C}ech  $1$-cocycle for $\widetilde{\Omega}^1_{\Theta}$. Besides,  the $\mbox{$\cup$-product}$ is given by formula~\eqref{cup-Cech}.
\end{nt}

\bigskip

We suppose now that $S$ is a separated scheme. Then the intersection of two open affine subschemes in $S$ is again an open affine subscheme.

For a $\oo((t))$-space $\Theta$ over $S$ we will  construct a natural determinantal $\oo_S^*$-gerbe ${\mathcal Det}(\Theta)$, which is a  locally connected sheaf of
$\oo_S^*$-groupoids with the descent properties (see  more on gerbes, for example, in~\cite[\S~1]{KV} and references therein,   and see also the similar   construction of  $\oo_S^*$-gerbe in~\cite[\S~3]{KV}).

First, we construct a  locally connected sheaf of
$\oo_S^*$-groupoids $\widetilde{{\mathcal Det}}(\Theta)$.
By definition, for an open $U \subset S$ the objects of the category $\widetilde{{\mathcal Det}}(\Theta) (U)$
are $\oo_U$-modules subsheaves $\RR$ of $\M |_U$ such that for any $i \in I $ (see formula~\eqref{good-cov}) and any  affine open subscheme $V \subset U_i \cap U$ we have
that $\phi_i (\RR (V)) $ coincides with the image of $\oo_S(V)[[t]]$ under some $\oo_S(V)$-module homeomorphism of $\oo_S(V)((t))$.
Now for any two objects $\RR_1$ and $\RR_2$ of  $\widetilde{{\mathcal Det}}(\Theta) (U)$ we define the $\oo_U^*$-torsor as the $\oo_U^*$-torsor that corresponds
to the line bundle given by the projective $\oo_S(V)$-module  ${ \det (\phi_i(\RR_1(V)) \mid  \phi_i(\RR_2(V)))}$ for any open affine  subscheme $V \subset U_i \cap U$ and any $i \in I$.
The set of all sections over $U$ of the constructed $\oo_U^*$-torsor is $\Hom (\RR_1, \RR_2) $ in the category $\widetilde{{\mathcal Det}}(\Theta) (U)$.

Thus we obtained a  locally connected sheaf of $\oo_S^*$-groupoids $\widetilde{{\mathcal Det}}(\Theta)$, which is a pre-stack. The corresponding associated stack is  ${\mathcal Det}(\Theta)$, which is  an $\oo_S^*$-gerbe. This gerbe does not depend on the choice of the covering $\{ U_i \}$ and the isomorphisms $\phi_i$
from the equivalence class of data.

By the gerbe ${\mathcal Det}(\Theta)$ we will construct its class in $\check{H}^2(S, \oo_S^*)$. Fix a data in~\eqref{good-cov}
such that all $U_i$ are affine open subschemes. For any $i \in I$ we denote ${\RR_i = \phi_i^{-1}(\oo_S(U_i)[[t]])}$, which is the object
from the category ${\mathcal Det}(\Theta) (U_i)$. By Proposition~\ref{d} and Corollary~\ref{Con}, for any $i, j \in I$ there is an element
$u_{ij }  \in \Hom (\RR_j |_{U_{i j}}, \RR_i |_{U_{i j}}) $. Now for any $i,j, k \in I$ we define an element from
$\oo_S^* (U_{i j k})$ by
\begin{equation}  \label{gerbe-2-cocycle}
h_{ijk} = u_{ik}^{-1} u_{ij} u_{jk}  \in \Hom (\RR_k |_{U_{i j k}}, \RR_k |_{U_{i j k}})  \, \mbox{.}
\end{equation}
This gives a \v{C}ech $2$-cocycle, and its class $[{\mathcal Det}(\Theta)]$ in  $\check{H}^2(S, \oo_S^*)$ depends only on $\Theta$ and does not depend on all the other choices  we have made (see the explanation in general case, e.g., in~\cite[\S~5.2]{Bry}).

\begin{nt} \label{specific} \em
We can take specific elements
$$
u_{ij}  \in \Hom (\RR_j |_{U_{i j}}, \RR_i |_{U_{i j}}) \simeq \det (\phi_{ij} (\oo_S(U_{ij})[[t]]) \mid \oo_S(U_{ij})[[t]])
$$
in formula~\eqref{gerbe-2-cocycle} that we used to obtain formula~\eqref{coc-D} and also used the dual to these elements in   the proof of Proposition~\ref{d}. These elements $u_{ij }$ depends on the choice of a data  in~\eqref{good-cov}.
\end{nt}

\medskip

\begin{Th}  \label{theor-4}
Let $\Theta$ be a $\oo((t))$-space over a separated scheme $S$ over $\Q$. Then in $\check{H}^2(S, \oo_S^*)$ we have an equality
$$
[{\mathcal Det}(\Theta)]^{12}  = \partial \, (c_1(\widetilde{\Omega}^1_{\Theta}) \cup c_1(\widetilde{\Omega}^1_{\Theta}) )    \, \mbox{.}
$$
(We use the multiplicative notation for the group law in $\check{H}^2(S, \oo_S^*)$.)
\end{Th}
\begin{proof}
Fix a  data in~\eqref{good-cov}
such that all $U_i$ are affine open subschemes. This data defines the transition automorphisms ${\phi_{ij} : \oo_S(U_{ij})((t))  \lrto \oo_S(U_{ij})((t))  }$
for any ${i, j \in I}$.  Now by these transition automorphisms we will construct the cochain map $P$ from the complex~\eqref{cochain-functor} when
 $F = \AutL_{\Q}$  and  $G = {\GG_m}_{\Q}$ to the \v{C}ech complex for the sheaf $\oo_S^*$ with respect to the cover $\{ U_i \}$. (We note that according to Remark~\ref{over_R} we put $C^0(\AutL_{\Q}, {\GG_m}_{\Q}) = \Q^*$ in complex~\eqref{cochain-functor}). For any $i_0, \ldots, i_k \in I$
 and any ${c \in C^k(\AutL_{\Q}, {\GG_m}_{\Q})}$ we define
 $$
 P(c) = c (\phi_{i_0 i_1} |_{U_{i_0 \ldots i_k}}  , \phi_{i_1 i_2}  |_{U_{i_0 \ldots i_k}}  , \ldots , \phi_{i_{k-1} i_k} |_{U_{i_0 \ldots i_k}} ) \,  \in \, \oo_S^*(U_{i_0 \ldots i_k} )
 $$
 when $k \ge 1$. For $k=0$ we put the obvious identity map. It is easy to see that the map $P$ is a cochain  map of complexes. Now directly from the construction we obtain that $P$ maps the $2$-cocycle $D$ to the \v{C}ech $2$-cocycle obtained from $\oo_S^*$-gerbe ${\mathcal Det}(\Theta)$ by formula~\eqref{gerbe-2-cocycle}, where  we  take specific $u_{ij}$  as in Remark~\ref{specific}.

Now for any $1$-cocycle $f$ from $C^1(\AutL_{\Q}, \, {L \GG_m}_{\Q})$ and any $2$-cocycle $g$ from $C^2(\AutL_{\Q}, \, {L \GG_m}_{\Q}
 \otimes {L \GG_m}_{\Q})$ we have that
 \begin{equation}  \label{map-1-coc}
 U_{ij} \longmapsto \phi_i^{-1}(f(\phi_{ij})) \, \mbox{,} \qquad U_{ijk} \longmapsto  \phi_i^{-1}(g (\phi_{ij}, \phi_{jk}))
 \end{equation}
 are \v{C}ech  $1$- and $2$-cocycles
 for the sheaves $\M^*$ and $\M^* \otimes_{\dz} \M^*$ on $S$ correspondingly. Moreover, from explicit formulas we have that these maps of cocycles commute with the $\cup$-products.

Under the first map in formula~\eqref{map-1-coc}   the $1$-cocycle $\tau$ (see formula~\eqref{form-tau})  goes to  the \v{C}ech $1$-cocycle that gives the sheaf $\widetilde{\Omega}^1_{\Theta}$
  when we take $d (\phi_i^{-1}(t))$ as local sections for this sheaf. Hence and from explicit formula~\eqref{FBT} for the $2$-cocycle $\hat{B}$ we obtain that $P (\hat{B})$ is the \v{C}ech $2$-cocycle
  given by formula~\eqref{expl}  (which is written for \v{C}ech  cocycles by Remark~\ref{dif-coc}).

  Now we apply Theorem~\ref{theor-3} to finish the proof of the theorem.

\end{proof}

\begin{nt}  \em
By constructions, the left hand side and the right hand side of formula from Theorem~\ref{theor-4}  can be considered as  \v{C}ech $2$-cocycles when we fix the cover by affine open subschemes $S = \bigcup_{i \in I} U_i$ and a data as in~\eqref{good-cov}.
Then formula from Theorem~\ref{theor-4}  will be true for \v{C}ech $2$-cocycles  and for any separated $S$-scheme (not only for a scheme over $\Q$) if the answer for the question in Remark~\ref{Quest-2} will be positive.
\end{nt}

\begin{nt}  \em
It follows from Remark~\ref{Nilp}  that if   $S$ is  a reduced scheme, then
the \v{C}ech $2$-cocycle $\partial \, (c_1(\widetilde{\Omega}^1_{\Theta}) \cup c_1(\widetilde{\Omega}^1_{\Theta})$ constructed by a data~\eqref{good-cov}  is trivial (see
Remark~~\ref{dif-coc}).

\end{nt}

\begin{nt}  \em
Let $\pi \, :  \, X \to S$ be a smooth morphism of relative dimension $1$ between separated schemes.
Let $i : S \to X$ be a section of $\pi$. Then $\pi$ is a separated morphism  and $D = i(S)$ is a closed subscheme of $X$. Moreover, by~\cite[Corollary 16.9.9, Theorem 17.12.1]{EGA4}, $D$  is a Cartier divisor  on $X$.
We consider the sheaf
$$\M = \varinjlim_{n \in \mathbb{Z}} \, \varprojlim_{m <n } \oo_X(nD)/ \oo_X(mD)$$
of rings on $D$.  This sheaf corresponds to the punctured formal neighbourhood of $D$ on $X$. Since $S \simeq D$, we consider $\M$ as a sheaf on $S$.  Now $\Theta= (S, \M)$ is an $\oo((t))$-space over~$S$.
But we can take a data in~\eqref{good-cov} for $\Theta$ such that  both explicit \v{C}ech $2$-cocycles   on $S$ that correspond to classes $[{\mathcal Det}(\Theta)]$
and ${\partial \, (c_1(\widetilde{\Omega}^1_{\Theta}) \cup c_1(\widetilde{\Omega}^1_{\Theta}) )}$ will be trivial. The reason is that $\M$
contains the  subsheaf
$$ \varprojlim_{m <0 } \oo_X/ \oo_X(mD) $$
that is locally isomorphic on $S$ to $\oo_S[[t]]$.
This defines some data as in~\eqref{good-cov}.
Therefore, by construction, $[{\mathcal Det}(\Theta)]$ will be trivial. Besides, $\partial \, (c_1(\widetilde{\Omega}^1_{\Theta}) \cup c_1(\widetilde{\Omega}^1_{\Theta}) )$ will be trivial, since the transition automorphisms $\phi_{ij}$ in Definition~\ref{def2} will
come from the continuous automorphisms of the algebra of Taylor series over a ring, and then in the construction of $\partial$ the Contou-Carr\`ere symbol $\CC$ will be trivial.
\end{nt}

\section*{Appendix A. Lie algebra valued functors constructed from group functors}
In this Appendix we collected some statements on how to construct the Lie algebra valued functor from a group functor with some conditions  and how to construct the $2$-cocycle  on the Lie algebra valued functor from a corresponding $2$-cocycle on a group functor.

 Above in the paper, we applied  these statements  to the group functors represented by ind-schemes.

We fix a commutative ring $R$. All functors  which we will consider will be from the category of commutative {$R$-algebras}.
By $A$ we  denote an arbitrary commutative $R$-algebra.

In sections~A.1 and A.2 below we mainly follow, but with some additions, work~\cite[Expose II]{SGA3} written by M.~Demazure.

\setcounter{equation}{0}
\renewcommand{\theequation}{A.\arabic{equation}}
\setcounter{nt}{0}
\renewcommand{\thent}{A.\arabic{equation}}
\setcounter{prop}{0}
\renewcommand{\theprop}{A.\arabic{equation}}

\subsection*{A.1 \; Tangent spaces functors}

Let $G$ be a  (covariant) functor. Let $x \in G(R)$. Then the tangent space functor $TG_x$ of $G$ at $x$ is defined as
$$
TG_{x}(A) = \rho^{-1}(x) \, \mbox{,}
$$
where $\rho \, : \, G(A[\varepsilon]/ \varepsilon^2)  \to G(A)$ is the natural map, and we consider the image of $x$ in~$G(A)$.

For any free $A$-module $V$ of finite rank we consider the commutative
 $A$-algebra $I_V$ that is isomorphic to $ A \oplus V$ as an $A$-module and  the multiplication in $I_V$ is defined by the facts that the first direct summand
  $A$ is  $A \cdot 1$ and  the second direct summand $V$ is an ideal of the algebra $I_V$ with $V^2 = (0)$.
  For any $x \in G(R)$ we denote
$$
G(I_V)_{x} = \varrho^{-1}(x)  \, \mbox{,}
$$
where $\varrho \, : \,   G(I_V)   \to G(A)  $ is the natural map.   We note that $G(I_A)_x = TG_{x}(A)$.

We suppose that $G$ satisfies the following condition for any free $A$-modules $V_1$ and $V_2$ of finite ranks
\begin{equation}  \label{cond1}
G(I_{V_1 \oplus V_2})_{x}  \xrightarrow{\sim} G(I_{V_1})_x  \times G(I_{V_2})_x  \, \mbox{,}
\end{equation}
where the map is induced by the natural homomorphisms $I_{V_1 \oplus V_2}  \to I_{V_i}$.

In this case  $TG_x$
has a natural structure of ${\mathbb A}^1_R$-module, where ${\mathbb A}^1_R(A) = A$ (see Appendix to Lecture~4 in~\cite{Mum}
 when $A =R$ is a field, and see~\cite[Expose II, Def.~3.5, Prop.~3.6]{SGA3}   in the general case that has the same form).
 For any fixed  $\xi_1, \xi_2 \in G(I_A)_x$ and any  $\alpha, \beta \in A$, the result $\alpha \xi_1 + \beta \xi_2$
 comes from the following diagram
$$
G(I_A)_x \times G(I_A)_x  \xleftarrow{\sim}  G (I_{A \oplus A})_x \xrightarrow{\langle \alpha, \beta \rangle}  G(I_A)_x \, \mbox{,}
$$
where  the map $\langle \alpha, \beta  \rangle $ is induced by the $A$-module homomorphism
$(\gamma, \delta) \mapsto (\alpha \gamma + \beta \delta)$ from $A \oplus A$ to $A$. The image of element $\xi_1 \times \xi_2$ is equal, by definition, to $\alpha \xi_1 + \beta \xi_2$.

\begin{nt}  \em
By {\em an ind-scheme} we mean an ind-object of the category of schemes such that all transition maps in the ind-object are closed embeddings of schemes.
If a functor $G$ is represented by an ind-scheme over $R$, then condition~\eqref{cond1} is satisfied, since it is satisfied for schemes. Moreover, in this case $TG_{x}$ is the inductive limit of  tangent space functors to corresponding  schemes, and the ${\mathbb A}^1_R$-module structure on $TG_x$
induced by this inductive limit coincides with the ${\mathbb A}^1_R$-module structure described above.
\end{nt}

Now we will consider only group functors and we will suppose that they satisfy condition~\eqref{cond1}.

We note  that for a group functor  $G$, it is enough to check   condition~\eqref{cond1} only for $x=e$, where $e \in G(R)$ is the identity element.
Besides, the group structure on $G$ induces the group structure on $TG_e$, and this group structure coincides with the group structure coming from the structure of ${\mathbb A}^1_R$-module defined above, see Corollary~3 from Proposition 3.9 in~\cite[Expose II]{SGA3}.

Now according to  Definition~4.6 from~\cite[Expose II]{SGA3}, we say that {\em a group functor $G$ is good}, if additionally to condition~\eqref{cond1}
the functor $G$ satisfies the following condition
\begin{equation}  \label{cond2}
TG_e (A) \otimes_A I_V  \xrightarrow{\sim}  TG_e (I_V)  \,  \mbox{,}
\end{equation}
where $A$ is any commutative $R$-algebra and $V$ is any free $A$-module of finite rank. The map is induced by the natural embedding $A \hookrightarrow I_V$
and the $I_V$-module structure on~$TG_e (I_V)$.

\begin{nt}  \em
If a group functor $G$ is represented by a group ind-scheme, then $G$ is good, because a functor  represented by a scheme
satisfies condition~\eqref{cond1} and condition~\eqref{cond2} that can be formulated for any  functors satisfying condition~\eqref{cond1} .
\end{nt}

\subsection*{A.2 \; Lie bracket}

For a good group functor $G$ we denote $\Lie G = TG_e$.

Now we consider only good group functors. For good group functors it is possible to define the Lie bracket on $\Lie G$ in the same way as it is done before Proposition~4.8
 in~\cite[Expose II]{SGA3}. We have the following commutative diagram with exact columns and rows.
$$
\xymatrix{
 \Lie G (A) \ar@{^{(}->}[r]   \ar@{^{(}->}[d]  & \Lie G (A[\ve_1]/ (\ve_1^2))  \ar@{->>}[r]  \ar@{^{(}->}[d]  & \Lie G (A) \ar@{^{(}->}[d] \\
 \Lie G (A[\ve_2]/ (\ve^2)) \ar@{^{(}->}[r]  \ar@{->>}[d] &  G (A[\ve_1, \ve_2]/ (\ve_1^2, \ve_2^2) )  \ar@{->>}[d]_(0.4){\alpha} \ar@{->>}[r]^(0.55){\beta} & G(A[\ve_2]/ (\ve_2^2))  \ar@{->>}[d] \\
 \Lie G (A)  \ar@{^{(}->}[r] &   G (A[\ve_1]/ (\ve_1^2))  \ar@{->>}[r]  & G(A)
 }
$$
Here $\Lie G (A) \subset G(A[\ve_1 \ve_2]/ (\ve_1^2 \ve_2^2))$ in the left upper corner is canonically isomorphic to $\Ker \alpha \cap \Ker \beta$.

Now we take an element $f_1 \in \Lie G (A)  \subset  G (A[\ve_1]/ (\ve_1^2)) $ in the left bottom corner of the diagram
and an element $f_2 \in \Lie G (A)  \subset G (A[\ve_2]/ (\ve_2^2)) $ in the right upper corner of the diagram.
Using that all columns and rows of the diagram have canonical group splittings, we obtain that $f_1, f_2 \in G(A[\ve_1, \ve_2] / (\ve_1^2, \ve_2^2))$.
By definition, {\em the Lie bracket }
\begin{equation}  \label{ff}
[f_1, f_2] = f_1 f_2 f_1^{-1}  f_2^{-1}  \in  G (A[\ve_1, \ve_2] / (\ve_1^2, \ve_2^2))
\end{equation}
is the image of the element from $\Lie G (A)$ in the upper left corner under the embedding induced by the embedding of rings
$A [\ve_1 \ve_2] / (\ve_1^2 \ve_2^2)   \hookrightarrow A [\ve_1, \ve_2]/ (\ve_1^2, \ve_2^2)$.

The described bracket
$$[\cdot, \cdot] \; : \; \Lie G (A) \times \Lie G (A)  \lrto \Lie G (A)$$
is functorial with respect to $A$,
$A$-bilinear   and satisfies the Jacobi identity (see formulas (5) and (6)  after Lemma~5.11 in~\cite[\S~5]{Mi} for commutators in a group).

By construction, we have that $[f_1, f_2] + [f_2, f_1]= 0$ for any $f_1, f_2 \in \Lie G (A)$.

By {\em an ind-affine ind-scheme} we mean an ind-scheme  ${M= \mbox{``$\varinjlim\limits_{i \in I}$''} \Spec A_i}$. By \linebreak
${\oo (M)= \varprojlim\limits_{i \in I}  A_i}$ we denote the corresponding topological algebra of regular functions on $M$ (with the topology of projective limit, taking the discrete topology on each $A_i$).

If $G$ is represented by an ind-affine ind-scheme, then
$[f,f]=0$ for any $f \in \Lie G (A)$, because in this case the $A$-algebra $\Lie G (A)$ is isomorphic to the Lie $A$-algebra of all continuous left-translation-invariant
$A$-derivations of the topological $A$-algebra $\oo (G_A)$ of  regular functions on the ind-scheme $G_A$. (See Proposition~3.13, Theorem 4.1.4 and \S~4.11
in~\cite[Expose II]{SGA3} for the case of right-translation-invariant
$A$-derivations, and the $A$-algebra $\Lie G (A)$ is anti-isomorphic to the Lie $A$-algebra of such continuous derivations.)

Thus, in case when $G$ is represented by an ind-affine ind-scheme, the bracket $[ \cdot, \cdot ]$ on $\Lie G$ defines the  Lie $A$-algebra structure on the $A$-module
$\Lie G (A)$ for any commutative {$R$-algebra~$A$}.

\begin{nt} \em
It is easy to see also that $[f_1, f_2]$ does not depend on the choice of liftings of $f_1, f_2$ to $G (A[\ve_1, \ve_2] / (\ve_1^2, \ve_2^2))$.
\end{nt}

\subsection*{A.3  \; $2$-cocycles}

We consider a central extension of group functors
$$
1 \lrto F \lrto \widetilde{G} \stackrel{\pi}{\lrto} G \lrto 1
$$
such that there is a section $\sigma$ of $\pi$, but we don't demand that $\sigma$ is a group section. We suppose that $F$ and $G$ are represented by group ind-affine ind-schemes.
Then   $\widetilde{G}$ is also represented by a group ind-affine ind-scheme (where we use that there is the section~$\sigma$). The section $\sigma$  gives the $2$-cocycle $\Lambda$ on $G$ with values in $F$.

We obtain the corresponding central extension  of Lie algebra valued functors
$$
0 \lrto \Lie F \lrto \Lie \widetilde{G}  \lrto \Lie G \lrto 0 \, \mbox{,}
$$
and $\sigma$ induces the section $ \Lie G  \lrto \Lie \widetilde{G}$. In the usual way (analogously to the case of central extensions of group functors in the proof of Proposition~\ref{Ext-funct}), this section gives the Lie algebra  $2$-cocycle
$$\Lie \Lambda  \;  :   \;  \Lie G \times \Lie G \lrto \Lie F \, \mbox{,}$$
which can be explicitly obtained from the cocycle $\Lambda$  in the following way.

Let $h_1, h_2$ be from $\Lie G (A)$, where $A$ is any commutative $R$-algebra. We consider elements $h_i$ (where $i=1$ and $i=2$)
as elements of $ G ( A[\ve_1, \ve_2]/ (\ve_1^2, \ve_2^2) )$ through the embeddings
$$
h_i \in \Lie G (A) \subset G(A[\ve_i]/(\ve_i^2) )   \subset G (A[\ve_1, \ve_2]/ (\ve_1^2, \ve_2^2) )  \, \mbox{.}
$$

Then we have that
$$
\Upsilon (h_1, h_2) = \Lambda (h_1, h_2) \cdot \Lambda (h_2, h_1)^{-1}
$$
is an element from the Lie $A$-algebra
$$\Lie F(A)  \subset F (A[\ve_1 \ve_2]/ (\ve_1^2 \ve_2^2))  \subset F(A[\ve_1, \ve_2]/ (\ve_1^2, \ve_2^2)) \mbox{.} $$
This follows at once from the diagram from Section A.2 written for the functor  $F$ and from the identity  $\Lambda(1, y)= \Lambda(y,1)= \Lambda(1,1)$, which follows from the cocycle identity for~$\Lambda$.

The proof of the following proposition is similar to the proof of Proposition 2.6.1 from~\cite{GR} given in the  case of central extension of Lie groups.
In our case we have to work over the ring $A[\ve_1, \ve_2]/ (\ve_1^2, \ve_2^2)$. Besides, analogs of one parameter subgroups are
obtained from the functors over subrings $A[\ve_i]/ (\ve_i^2)$, and an analog of the mixed second-order partial derivative is the functor over the subring
$A[\ve_1 \ve_2]/(\ve_1^2 \ve_2^2)$.

\begin{prop}  \label{ex-form-coc}
We have that $\Lie \Lambda (h_1, h_2) = \Upsilon (h_2, h_2)$ for any $h_1, h_2$  from $\Lie G (A)$.
\end{prop}

\vspace{0.3cm}

\noindent Steklov Mathematical Institute of Russsian Academy of Sciences, 8 Gubkina St., Moscow 119991, Russia,
 {\em and}

 \noindent National Research University Higher School of Economics, Laboratory of Mirror Symmetry,  6 Usacheva str., Moscow 119048, Russia,
{\em and}

\noindent National University of Science and Technology ``MISiS'',  Leninsky Prospekt 4, Moscow  119049, Russia

\noindent {\it E-mail:}  ${d}_{-} osipov@mi{-}ras.ru$


\begin{thebibliography}{99}

\bibitem{BBE} A.~Beilinson, S.~Bloch, H.~Esnault, {\em $\epsilon$-factors for Gauss-Manin determinants.} Dedicated to Yuri I. Manin on the occasion of his 65th birthday. Mosc. Math. J. 2 (2002), no. 3, 477--532.


\bibitem{Bl} S.~Bloch, {\em  $K_2$ and algebraic cycles.} Ann. of Math. (2) 99 (1974), 349--379.

\bibitem{B1}
R. Bott, {\em On the characteristic classes of groups of diffeomorphisms.} Enseign. Math. (2) 23 (1977), no. 3--4, 209--220.


\bibitem{B2}
 R.~Bott, {\em On some formulas for the characteristic classes of group-actions.} Differential topology, foliations and Gelfand-Fuks cohomology (Proc. Sympos., Pontifícia Univ. Católica, Rio de Janeiro, 1976), pp. 25--61, Lecture Notes in Math., 652, Springer, Berlin, 1978.


\bibitem{BKTV} P.~Bressler,  M.~Kapranov, B.~Tsygan, E.~Vasserot, {\em Riemann-Roch for real varieties.} Algebra, arithmetic, and geometry: in honor of Yu. I. Manin. Vol. I, 125--164, Progr. Math., 269, Birkhäuser Boston, Boston, MA, 2009.




\bibitem{Bro}
K. S. Brown, {\em Cohomology of groups.} Graduate Texts in Mathematics, 87. Springer-Verlag, New York-Berlin, 1982.


\bibitem{Bry}
J.-L. Brylinski,  {\em Loop spaces, characteristic classes and geometric quantization.} Progress in Mathematics, 107. Birkhäuser Boston, Inc., Boston, MA, 1993.


\bibitem{CC1} C. Contou-Carr\`{e}re, {\em Jacobienne locale, groupe de bivecteurs de Witt universel, et symbole mod\'{e}r\'{e}.} (French)  C. R. Acad. Sci. Paris S\'{e}r. I Math. 318 (1994), no. 8, 743--746.


\bibitem{CC2} C. Contou-Carr\`{e}re, {\em
Jacobienne locale d'une courbe formelle relative.} (French)  Rend. Semin. Mat. Univ. Padova 130 (2013), 1--106.

\bibitem{D1}
P.~Deligne, {\em
Le d\'{e}terminant de la cohomologie.} (French)  Current trends in arithmetical algebraic geometry (Arcata, Calif., 1985), 93--177,
Contemp. Math., 67, Amer. Math. Soc., Providence, RI, 1987.

\bibitem{D2}
P.~Deligne, {\em Le symbole mod\'{e}r\'{e}.} (French)  Inst. Hautes \'{E}tudes Sci. Publ. Math. No. 73 (1991), 147--181.





\bibitem{SGA3}
M.\,Demazure, A.\,Grothendieck, {\it Sch\'emas en groupes. I: Prorpi\'et\'es g\'en\'erales des sch\'emas en groupes, S\'eminaire de G\'eom\'etrie Alg\'ebrique du Bois Marie 1962/64 (SGA 3)}, Lecture Notes in Mathematics, {\bf 151} Springer-Verlag, Berlin-New York (1970).



\bibitem{Die}
J.~Dieudonn\'e, {\em Introduction to the theory of formal groups}, Pure and Applied Mathematics, 20. Marcel Dekker, Inc., New York, 1973.


\bibitem{FF} B. L.  Feigin, D. B. Fuks, {\em  Cohomology of Lie groups and Lie algebras.} (Russian) Current problems in mathematics. Fundamental directions, Vol. 21, 121--209,  Itogi Nauki i Tekhniki, Akad. Nauk SSSR, Vsesoyuz. Inst. Nauchn. i Tekhn. Inform., Moscow, 1988. English translation: B. L. Feigin and D. B. Fuks, {\em Cohomologies of Lie Groups and Lie Algebras,}
     Lie groups and Lie algebras, II, Encyclopaedia Math. Sci., 21, Springer, Berlin,
2000, 125--223.

\bibitem{FZh} E.~Frenkel, X.~Zhu, {\em  Gerbal representations of double loop groups.} Int. Math. Res. Not. IMRN 2012, no. 17, 3929--4013.


\bibitem{Fu} D. B. Fuks, {\em  Cohomology of infinite-dimensional Lie algebras.} Translated from the Russian by A. B. Sosinskii. Contemporary Soviet Mathematics. Consultants Bureau, New York, 1986.




\bibitem{GO1}
 S. O. Gorchinskiy, D. V. Osipov,  {\em A higher-dimensional Contou-Carr\`{e}re symbol: local theory.} (Russian) Mat. Sb. 206 (2015), no. 9, 21--98; translation in Sb. Math. 206 (2015), no. 9--10, 1191--1259.


\bibitem{GO2} S. O. Gorchinskiy, D. V. Osipov, {\em Continuous homomorphisms between algebras of iterated Laurent series over a ring.} (Russian) Tr. Mat. Inst. Steklova 294 (2016), Sovremennye Problemy Matematiki, Mekhaniki i Matematicheskoy Fiziki. II, 54--75. English version published in Proc. Steklov Inst. Math. 294 (2016), no. 1, 47--66.

\bibitem{GO3} S. O. Gorchinskiy, D. V. Osipov, {\em Higher-dimensional Contou-Carrère symbol and continuous automorphisms.} (Russian) Funktsional. Anal. i Prilozhen. 50 (2016), no. 4, 26--42; translation in Funct. Anal. Appl. 50 (2016), no. 4, 268--280.

\bibitem{GO4} S. O. Gorchinskiy, D. V. Osipov, {\em
The higher-dimensional Contou-Carr\`ere symbol and commutative group schemes.} (Russian) Uspekhi Mat. Nauk 75 (2020), no. 3(453), 185--186;
English translation in Russian Math. Surveys, 75:3 (2020), 572--574.


\bibitem{GO5} S. O. Gorchinskiy, D. V. Osipov, {\em
Iterated Laurent series over rings and Contou-Carr\`re symbol.} (Russian) Uspekhi Mat. Nauk 75 (2020), no. 6(456), 3--84;
English translation in Russian Math. Surveys, 75:6 (2020), 995--1066.

\bibitem{EGA4}
A.\,Grothendieck, {\it \'El\'ements de g\'eom\'etrie alg\'ebrique. IV. \'Etude locale des sch\'emas et des morphismes de sch\'emas IV}, Inst. Hautes \'Etudes Sci. Publ. Math., {\bf 32} (1967).


\bibitem{GR}
L. Guieu, C.  Roger, {\em  L'alg\`{e}bre et le groupe de Virasoro.} (French)  Aspects g\'{e}om\'{e}triques et alg\'{e}briques, g\'{e}n\'{e}ralisations.  With an appendix by Vlad Sergiescu. Les Publications CRM, Montreal, QC, 2007.




\bibitem{KP}
 V. G. Kac, D. H.  Peterson,  {\em Spin and wedge representations of infinite-dimensional Lie algebras and groups.} Proc. Nat. Acad. Sci. U.S.A. 78 (1981), no. 6, part 1, 3308--3312.



\bibitem{KV}
 M. Kapranov, \'{E}. Vasserot,  {\em Formal loops. II. A local Riemann-Roch theorem for determinantal gerbes.} Ann. Sci. \'{E}cole Norm. Sup. (4) 40 (2007), no. 1, 113--133.


\bibitem{Ka}
K. Kato, {\em Milnor K-theory and the Chow group of zero cycles.} Applications of algebraic K-theory to algebraic geometry and number theory, Part I, II (Boulder, Colo., 1983), 241--253,
Contemp. Math., 55, Amer. Math. Soc., Providence, RI, 1986.


\bibitem{KW}
B. Khesin, R. Wendt, {\em The geometry of infinite-dimensional groups.} Ergebnisse der Mathematik und ihrer Grenzgebiete. 3. Folge. A Series of Modern Surveys in Mathematics [Results in Mathematics and Related Areas. 3rd Series. A Series of Modern Surveys in Mathematics], 51. Springer-Verlag, Berlin, 2009.

\bibitem{Mi}  J. Milnor, {\em  Introduction to algebraic K-theory.} Annals of Mathematics Studies, No. 72. Princeton University Press, Princeton, N.J.; University of Tokyo Press, Tokyo, 1971.

\bibitem{M}
J.~Morava, {\em  An algebraic analog of the Virasoro group.} Quantum groups and integrable systems (Prague, 2001). Czechoslovak J. Phys. 51 (2001), no. 12, 1395--1400.


 \bibitem{Mum} D.~Mumford, {\em Lectures on curves on an algebraic surface.} With a section by G. M. Bergman. Annals of Mathematics Studies, No. 59 Princeton University Press, Princeton, N.J. 1966.


\bibitem{Osip1}
 D. V. Osipov,  {\em Adele constructions of direct images of differentials and symbols.} (Russian) Mat. Sb. 188 (1997), no. 5, 59--84; translation in Sb. Math. 188 (1997), no. 5, 697--723.

\bibitem{OZ}
D.~Osipov, X.~Zhu, {\em  The two-dimensional Contou-Carr\`{e}re symbol and reciprocity laws.}, J. Algebraic Geom. 25 (2016), no. 4, 703--774.



\bibitem{PS} A.~Pressley, G.~Segal, {\em
Loop groups},
Oxford Mathematical Monographs, Oxford Science Publications, The Clarendon Press, Oxford University Press, New York, Revised ed. edition 1988.


\bibitem{Se} G.~Segal, {\em Unitary representations of some infinite-dimensional groups.} Comm. Math. Phys. 80 (1981), no. 3, 301--342.




\end{thebibliography}
\end{document}